\documentclass{amsart}
\usepackage{mathtools, tikz-cd, stmaryrd, subcaption, amsthm, float, hyperref, enumerate, amssymb}

\usepackage[authoryear,longnamesfirst]{natbib}

\definecolor{sagegreen}{rgb}{0.05,0.53,0.26}
\hypersetup{citecolor = sagegreen, linkcolor = sagegreen, menucolor = sagegreen, colorlinks = true}

\usepackage{tikz}
\usetikzlibrary{arrows.meta, decorations.markings}

\DeclareMathOperator{\hor}{hor}
\DeclareMathOperator{\image}{im}
\DeclareMathOperator{\diff}{Diff}
\DeclareMathOperator{\id}{id}
\DeclareMathOperator{\Hol}{Hol}
\DeclareMathOperator{\pr}{pr}
\DeclareMathOperator{\dom}{dom}
\DeclareMathOperator{\ev}{ev}

\DeclarePairedDelimiterX{\corners}[1]{\ulcorner}{\urcorner}{\mkern-2mu #1 \mkern-2mu}

\newcommand{\smooth}{\mathcal{C}^\infty}

\newcommand{\R}{\mathbb{R}}
\renewcommand{\P}{\mathcal{P}}
\newcommand{\Pin}{\hat{\mathcal{P}}}

\newcommand{\Holonomy}{\textsf{Hol}}
\newcommand{\BundleConnection}{\textsf{B}\mathcal A_\bullet}

\newcommand{\inverse}{^{-1}}
\newcommand{\se}{\subseteq}

\newtheorem{definition}{Definition}
\newtheorem{remark}{Remark}

\newtheorem{theorem}{Theorem}
\newtheorem*{theorem*}{Theorem}
\newtheorem*{convention*}{Convention}
\newtheorem{lemma}[theorem]{Lemma}
\newtheorem{proposition}[theorem]{Proposition}

\newtheorem{corollary}[theorem]{Corollary}

\def\shadowbox{\hbox{\rule[0.05ex]{0.1ex}{1.15ex}%
    \hspace{-0.0ex}\rule[-0.0ex]{1.2ex}{0.1ex}%
    \hspace{0.0ex}\rule[-0.0ex]{0.1ex}{1.2ex}\hspace{-1.3ex}%
    \rule[1.15ex]{1.25ex}{0.1ex}\hspace{-0.0ex}\rule[-0.25ex]{0.3ex}{1.1ex}%
    \hspace{-1.2ex}\rule[-0.25ex]{1.1ex}{0.25ex}}}
    \def\qed{\ifmmode \hbox{\hfill\shadowbox}
        \else \hphantom{x}\hfill\shadowbox \fi}

\title{A Diffeological Construction of Singer's Universal Connection}
\author{Dion Mann}

\begin{document}

\begin{abstract}
    We provide a rigorous construction of I.M. Singer's \textit{universal connection}, a natural connection on a bundle of paths associated to any manifold, using the theory of diffeology. Furthermore, we generalize the universal connection to the diffeological setting, which enables the reconstruction of diffeological principal bundles with connections from their holonomy representations, extending known bundle-reconstruction results into diffeology. In particular, any two diffeological bundle-connection pairs with conjugate holonomy representations must be isomorphic as bundles, with the isomorphism taking horizontal paths to horizontal paths. Finally, we show that reconstruction and holonomy are functorial and, for a fixed diffeological group $G$, form an adjoint equivalence between the category of smooth homomorphisms from based loop groups into $G$ and the category of pointed diffeological principal $G$-bundle-connection pairs, extending to the diffeological setting the known equivalences for manifolds and Lie groups.
\end{abstract}

\maketitle
\tableofcontents

\section{Introduction}
\label{sec:Introduction}

Let $(M, \bullet)$ be a pointed, path-connected smooth manifold. In \cite{Singer}, I.M. Singer develops a so-called \textit{universal connection} associated to $(M, \bullet)$ as follows. The collection of piecewise-smooth, retrace-equivalent paths $\alpha : [0,1] \to M$ starting at the fixed point (that is, $\alpha(0) = \bullet)$ is denoted by $\mathcal F_\bullet(M)$. We have the \textit{target map} $\tau : \mathcal F_\bullet(M) \to M$ given by $\tau(\alpha) = \alpha(1)$. For each $x \in M$, the fiber $\tau\inverse(x)$ is the set of all paths in $M$ starting at $\bullet$ and ending at $x$. Let $\Omega_\bullet(M) = \tau\inverse(\bullet)$. Singer makes the observation that $\tau : \mathcal F_\bullet(M) \to M$ is something like a ``principal $\Omega_\bullet(M)$-bundle,'' which has a natural horizontal lifting function: for a path $\alpha : [0,1] \to M$ and an initial ``point'' $\beta_0 \in \tau\inverse(\alpha(0))$, we set $\bar \alpha_{\beta_0}(s) : [0,1] \to M$ to be the path in $M$ which first follows $\beta_0$ and then follows $\alpha$ until $\alpha(s)$. In this manner, we obtain a path $\bar \alpha_{\beta_0} : [0,1] \to \mathcal F_\bullet(M)$ into the total space which is the ``horizontal lift'' of $\alpha$. See Figure \ref{fig:SingerHorizontalLift}.

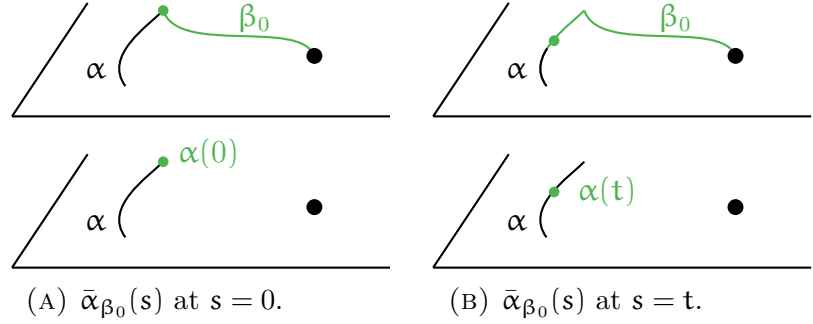
\begin{figure}
    \centering
    \definecolor{liftgreen}{RGB}{76,175,80}
    \newcommand{\SingerPlane}{%
        \draw (0,0) -- (8.8,0) (0,0) -- (1.75,2.65);
        \coordinate (a0) at (3.55,2.45); % alpha(0)
        \coordinate (at) at (2.83,1.75); % alpha(t)
        \coordinate (a1) at (2.66,0.68); % alpha(1)
        \coordinate (bp) at (7.10,1.40); % the basepoint
        \coordinate (bm) at (5.60,1.86); % some point on beta_0
    }%
    \tikzset{
        singer/.style={x=0.75cm, y=0.75cm, thick, line cap=round, line join=round},
        oriented/.style={postaction={decorate}, decoration={markings,
            mark=at position #1 with {\arrow{Stealth[length=2.3mm, width=1.8mm]}}}},
        oriented/.default=0.5,
        beta0/.style={insert path={
            (bp) .. controls +(-0.20,0.25) and +(0.90,0) .. (bm)
                 .. controls +(-1.10,0)    and +(0.22,-0.52) .. (a0)}},
        alpha 0t/.style={insert path={
            (a0) .. controls +(-0.25,-0.22) and +(0.15,0.20) .. (at)}},
        alpha t1/.style={insert path={
            (at) .. controls +(-0.375,-0.50) and +(-0.15,0.30) .. (a1)}},
        % points
        basepoint/.style={circle, fill=black, inner sep=0pt, minimum size=2.6mm,
            label={[font=\small, label distance=1pt]below:basepoint $\bullet$}},
        green dot/.style={circle, fill=liftgreen, inner sep=0pt, minimum size=1.8mm},
    }%
    % ------------------------------------------------------------- (A)
    \begin{subfigure}[t]{0.49\textwidth}
        \vspace{0pt}
        \centering
        \begin{tikzpicture}[singer, scale=0.85, transform shape]
            % top:
            \SingerPlane
            \draw[alpha 0t];
            \draw[oriented=0.6, alpha t1];
            \draw[liftgreen, oriented=0.45, beta0];
            \node[basepoint] at (bp) {};
            \node[green dot] at (a0) {};
            \node at (1.97,1.05) {$\alpha$};
            \node[liftgreen] at (5.70,2.28) {$\beta_0$};
            % bottom:
            \begin{scope}[yshift=-3.6cm*0.75]
                \SingerPlane
                \draw[alpha 0t];
                \draw[oriented=0.6, alpha t1];
                \node[basepoint] at (bp) {};
                \node[green dot] at (a0) {};
                \node at (1.97,1.05) {$\alpha$};
                \node[liftgreen, right=5pt] at (a0) {$\alpha(0)$};
            \end{scope}
        \end{tikzpicture}
        \caption{$\bar \alpha_{\beta_0}(s)$ at $s = 0$. In particular, this is just the path $\bar \alpha_{\beta_0}(0) = \beta_0$ in $M$.}
    \end{subfigure}%
    \hfill
    % ------------------------------------------------------------- (B)
    \begin{subfigure}[t]{0.49\textwidth}
        \vspace{0pt}
        \centering
        \begin{tikzpicture}[singer, scale=0.85, transform shape]
            % top:
            \SingerPlane
            \draw[oriented=0.6, alpha t1];
            \draw[liftgreen, oriented=0.45, beta0];
            \draw[liftgreen, oriented=0.6, alpha 0t];
            \node[basepoint] at (bp) {};
            \node[green dot] at (at) {};
            \node at (1.97,1.05) {$\alpha$};
            \node[liftgreen] at (5.70,2.28) {$\beta_0$};
            % bottom:
            \begin{scope}[yshift=-3.6cm*0.75]
                \SingerPlane
                \draw[alpha 0t];
                \draw[oriented=0.6, alpha t1];
                \node[basepoint] at (bp) {};
                \node[green dot] at (at) {};
                \node at (1.97,1.05) {$\alpha$};
                \node[liftgreen, right=6pt] at (at) {$\alpha(t)$};
            \end{scope}
        \end{tikzpicture}
        \caption{$\bar \alpha_{\beta_0}(s)$ at $s = t$. In particular, this is the path following $\beta_0$ first then $\alpha$ from $\alpha(0)$ to $\alpha(t)$, where the corner represents their join.}
    \end{subfigure}
    \caption{Singer's Universal Connection. The green paths represent the element of $\mathcal F_\bullet(M)$ determined by $\bar \alpha_{\beta_0}(s)$, which project down to the green point $\alpha(s)$ in $M$ via the target map $\tau : \mathcal F_\bullet(M) \to M$. Note that $\tau \circ \bar \alpha_{\beta_0} = \alpha$.}
    \label{fig:SingerHorizontalLift}
\end{figure}

An application of Singer's universal connection is to prove a certain re\-con\-struc\-tion-style theorem of Kobayashi:
\begin{theorem*}[\cite{KobayashiConnexionI}; see also \cite{Barrett}]
    Let $(M, \bullet)$ be a pointed smooth manifold which is path-connected and let $G$ be a Lie group. Let
    \[
    H : \Omega_{\bullet}(M) \to G
    \]
    be a given homomorphism which is \textit{smooth} in the sense that $u \mapsto H(\gamma_u)$ is a smooth map $\R^n \supseteq U \to G$ for every smooth family $\{\gamma_u\}_{u \in U}$ of loops based at $\bullet$, i.e. every smooth map $U \times [0,1] \to M$, $(u, t) \mapsto \gamma_u(t)$, with $U \se \R^n$ open. Then there exists a pointed principal $G$-bundle $\pi : (E, \xi_\bullet) \to (M, \bullet)$ and a connection $A$ on $\pi : E \to M$ such that the holonomy representation $H^A_{\xi_\bullet} : \Omega_{\bullet}(M) \to G$, defined implicitly by the equation
    \[
    \bar \gamma_{\xi_\bullet}(1) = \xi_\bullet \cdot H^A_{\xi_\bullet}(\gamma), \quad \text{for all } \gamma \in \Omega_{\bullet}(M),
    \]
    where $\bar \gamma_{\xi_\bullet}$ is the horizontal lift of $\gamma$ starting at $\xi_\bullet \in \pi\inverse(\bullet)$, realizes $H$; that is, $H^A_{\xi_\bullet} = H$.
\end{theorem*}
In Singer's approach, the bundle-connection pair in Kobayashi's theorem is the associated bundle $E = \mathcal F_{\bullet}(M) \times_H G$ with connection induced by the universal connection.

Kobayashi's notes (\citet{KobayashiConnexionI,KobayashiConnexionII}) are short announcements without proofs, and, as \citet[p.~1179]{Barrett} observes, their axioms ``do not seem to include any mention of differentiability''; Singer quotes the theorem for continuous homomorphisms \citep[p.~272]{Singer}. Complete proofs, for smooth principal bundles over manifolds and with a smoothness hypothesis on the holonomy map, were given by \citet{Barrett} and by \citet{CaetanoPicken}, and many variants have appeared since, for example in \citet{Singer}, \citet{gibilisco}, \citet{SchreiberWaldorf09} and \citet{CLW}; we discuss these works more toward the end of the introduction. The smoothness condition on $H$ stated above is the one used by \citet{CaetanoPicken} (also \citealp{Barrett}) and it says precisely that $H$ is smooth for the diffeology on $\Omega_\bullet(M)$ used in this paper (cf. \citealp[Prop.~A.6]{SchreiberWaldorf09}). This is the first hint as to why diffeology is a good framework for treating bundles involving path-like data.

We will follow Singer's construction of his \textit{universal connection}\footnote{The term \textit{universal connection} has been established in the literature since \citet{NarasimhanRamanan}, for the canonical connection on a Stiefel bundle: roughly that every connection on a principal bundle pull-back of it. Singer's universal connection is a different gadget and is universal in a different sense. This is made precise in Section \ref{sec:Kobayashi} and Remark \ref{rem:ImageOfK}.} in \citet{Singer}. Although the argument is certainly convincing enough, Singer seems to admit that some aspects are only given heurestically \citep[\S3]{Singer}. This is our first goal: to provide a fully rigorous and complete framework for Singer's universal connection using the theory of \textit{diffeology}, whose background we recall in Section \ref{sec:BackgroundInDiffeology}.

In accomplishing the above, we find that Singer's notions generalize naturally to the category of diffeological spaces and are not just restricted to smooth manifolds:
\begin{theorem*}[Diffeological Bundle Reconstruction]
    Let $(X, \bullet)$ be a pointed, path-connected diffeological space and $G$ a diffeological group. If $H : \Omega_\bullet(X) \to G$ is a smooth group homomorphism, then there exists a diffeological principal $G$-bundle equipped over $X$ with a diffeological connection $\Theta^H$ whose holonomy representation coincides with $H$.
\end{theorem*}

As a corollary, we also have the following holonomy-based classification of diffeological bundles:
\begin{theorem*}
    Let $\pi : E \to X$ and $\pi' : E' \to X$ be diffeological principal $G$-bundles. If there exists diffeological connections $A$ and $A'$ on the respective bundles $\pi$ and $\pi'$ whose holonomy representations agree up to conjugation, then there is a $G$-bundle isomorphism from $\pi$ to $\pi'$ that takes $A$ to $A'$.
\end{theorem*}

We also show the diffeological bundle reconstruction is reversible in the following sense. For a fixed diffeological group $G$, the collection of all triples $(X, x_0, H)$, where $(X, x_0)$ is a pointed, path-connected diffeological space and $H : \Omega_{x_0}(X) \to G$ is a smooth group homomorphism from its group of retrace-equivalent loops based at $x_0$, forms a category $\Holonomy^G$ which we can call the \textit{holonomy category}. The diffeological bundle reconstruction theorem then defines a functor $\Holonomy^G \to \BundleConnection^G$, where $\BundleConnection^G$ is the category of diffeological principal bundle-connection pairs, which turns out to be left adjoint to the holonomy functor $\BundleConnection^G \to \Holonomy^G$, i.e. the functor that forgets the total space and gives the holonomy representation homomorphism associated to a connection. The details are spelled out in Section \ref{subsec:CategoryEquivalence}, in which we then show:
\begin{theorem*}
    There is an equivalence of categories \upshape 
    \[
    \Holonomy^G \simeq \BundleConnection^G.
    \]
\end{theorem*}

\textbf{Related Works.} The reconstruction of bundle-connection pairs from holonomy has a long history. We can only indicate briefly what is known, without making any claim to provide a full review, and then state what is new here. 

Regarding the reconstruction theorem, \citet{gibilisco} obtains a bijective correspondence between bundle-connection pairs up to isomorphism and smooth representations of the resulting loop group up to conjugation \citep[Thm.~7.1]{gibilisco}. The term \emph{bundle-connection pair}, which we use in Section \ref{subsec:CategoryEquivalence}, is taken from \citet{gibilisco}. \citet{SchreiberWaldorf09} replace based loops by the thin path groupoid $\mathcal P_1(M)$ of a smooth manifold $M$, whose morphisms are thin homotopy classes of smooth paths with sitting instants. They show that principal $G$-bundles with connection over $M$, for a Lie group $G$, form a category equivalent to the category of \emph{transport functors} on $\mathcal P_1(M)$ \citep[Thm.~5.4]{SchreiberWaldorf09}. Their smoothness conditions are diffeological \citep[Props.~A.6]{SchreiberWaldorf09}: yet another hint for using diffeology in the setting of the present work. \citet{CLW} extend the results of \citet{Barrett,CaetanoPicken,SchreiberWaldorf09} from manifolds to differentiable stacks, with $G$ a Lie group. For a connected manifold $M$, their argument contains the thin homotopy analogues of our path bundle and of the isomorphism $\Phi$ of Theorem \ref{thm:Reconstruction}.

In virtually every work regarding the reconstruction theorem, one has to deal with how paths should be treated. There are many inequivalent treatments. \citet{Barrett} worked with thin-homotopy. \citet{CaetanoPicken} work instead with smooth loops with sitting instants, modulo \emph{intimacy} (a homotopy of rank at most one), which simplifies Barrett's arguments. The notion of \emph{retrace-equivalence}, which the present work uses, goes back to Kobayashi himself \citep[fn.~7]{Meneses}. Singer's path bundle also uses a version of retracing, which he called ``backtracking'' instead \citep[p.~272]{Singer}. Barrett observes that his arguments would also work with the loops that reduce to a trivial loop by finitely many reparameterizations and retracings \citep[p.~1205]{Barrett}. We discuss some advantages of using \emph{retrace-equivalence} later in Section \ref{sec:PathSpaces}.

Speaking of paths: path spaces of manifolds are not manifolds, and generalized smooth spaces have been used to handle them since Chen introduced his differentiable spaces in order to develop iterated integrals on path and loop spaces \citep{Chen73, Chen77}. A similar type of generalized smooth space, called ``diffeological spaces'', were introduced by \citet{Souriau}, and the theory of diffeological fiber bundles and connections that we use is due to \citet{PIZ}. See \citet{BaezHoffnung} for a comparison of Chen spaces and diffeological spaces. In diffeology, bundles and path-like data have so far been studied in the following sense. \citet{MagnotWatts} equip Milnor's classifying space $EG \to BG$ with a diffeology, for a diffeological group $G$, and study several analogues to the classical setting. They also construct a connection on $EG \to BG$, and ask a universal-type question of whether every diffeological connection is pulled back from it \citep[Question~3]{MagnotWatts}. \citet{Minichiello} shows that, for an arbitrary diffeological space $X$ and diffeological group $G$, diffeological principal $G$-bundles over $X$ are essentially the same as principal $\infty$-bundles \citep[Thm.~6.8 and Cor.~6.9]{Minichiello}; connections are not considered there. Both works use Iglesias-Zemmour's notion of a diffeological principal bundle and classify bundles in the spirit of the classical theory and of \citet{NarasimhanRamanan}. Our universal object instead lives over a fixed base $X$ and classifies bundles \emph{with} connection, for every diffeological group $G$, through their holonomy.

For manifolds and Lie groups, the reconstruction and classification of bundle-connection pairs from holonomy are therefore known, in several formulations \citep{Barrett, CaetanoPicken, gibilisco, SchreiberWaldorf09, CLW}, and so is the idea of building the reconstructed bundle as a quotient of a space of paths. For manifolds, the novelty of the present work lies in its method, namely Singer's universal connection made rigorous. Its new content lies in the generality, and we single out the following points:
\begin{enumerate}[(i)]
    \item The base $X$ is an arbitrary, path-connected diffeological space, and the structure group $G$ is an arbitrary diffeological group. Connections are understood in the sense of Iglesias-Zemmour with mild modifications, as smooth horizontal projections of paths, so no Lie algebra is needed.
    \item Singer's path bundle $\tau : \mathcal F_\bullet(X) \to X$ itself becomes a diffeological principal $\Omega_\bullet(X)$-bundle with a diffeological connection $\Theta$ whose horizontal lifts are Singer's (Proposition \ref{prop:HorizontalLiftOfUniversalConnection}). The reconstructed bundle-connection pair is then literally an associated bundle with the induced connection. This turns Barrett's analogy and Singer's heuristic description into theorems, and Theorem \ref{thm:CategoryEquivalence} makes precise the sense in which $(\tau, \Theta)$ is universal (Remark \ref{rem:ImageOfK}). In the thin homotopy setting over manifolds, the closest construction is the one of \citet{CLW} described above.
    \item We use retrace equivalence instead of thin homotopy. It is not known to us if there exists a \textit{diffeological} notion of thin-homotopy for which the holonomy of an arbitrary diffeological connection is invariant under it.
    \item The holonomy data are smooth homomorphisms on based loop groups, and the morphisms of $\Holonomy^G$ are pointed smooth maps between possibly different bases, together with a conjugating element of $G$. The resulting adjoint equivalence $\Holonomy^G \simeq \BundleConnection^G$ therefore incorporates the classification up to conjugation (Corollary \ref{cor:Classification}), which relates to the functorial formulations of \citet{SchreiberWaldorf09} and \citet{CLW}.
\end{enumerate}

\textbf{Outline of paper.} In Section \ref{sec:BackgroundInDiffeology}, we briefly provide the necessary background theory of diffeology and discuss diffeological bundles and connections. We then carefully set up retrace-equivalent path spaces and Singer's universal connection using the diffeological framework in Section \ref{sec:PathSpaces}, followed by its generalization to arbitrary diffeological spaces. Finally, we apply these constructions to prove our main results in Section \ref{sec:Kobayashi}.

\textbf{Acknowledgements.} I thank my Ph.D. advisor, Dr. Ambar Sengupta, for his support and guidance, as well as my wife, Dr. Heidi Benham, and my colleagues, Jacob Brown, Garett Cunningham, Benjamin Grant, Yukta Lodha, and Samuel Stuart for their meaningful discussions with me about various aspects of this paper.  Thanks also to Dr. Patrick Iglesias-Zemmour and Dr. Yael Karshon for their helpful comments on the first draft of this paper. Finally, I express my utmost gratitude for the anonymous referee's genuine and thorough review, suggestions, and questions, which has substantially improved both the mathematics and narrative of the manuscript.

\section{Background in Diffeology}
\label{sec:BackgroundInDiffeology}

We collect here the background on diffeology needed in this paper. The standard reference is the book \citet{PIZ}, which is organized into numbered articles; we cite these as ``art.~$n.m$'' for (non-negative) integers $n$ and $m$. Section \ref{subsec:BasicsOfDiffeology} briefly reviews the basic definitions, all of which may be found in \citet[Ch.~1]{PIZ}. The reader who is already acquainted with diffeology may skip to Section \ref{subsec:BundlesInDiffeology}, where we discuss diffeological fiber bundles and principal bundles following \citet[Ch.~8]{PIZ}. At that point, some of our conventions will differ from those of the book, and we point these out as they arise. Finally, Section \ref{subsec:ConnectionsAndHolonomyInDiffeology} discusses \textit{connections} and \textit{holonomy} in diffeology. These notions were introduced in \citet[arts.~8.32--8.35]{PIZ}; we modify the definition of a connection slightly (Remarks \ref{rem:DifferentConnectionDefinition} and \ref{rem:ProjectionAxiom}), while giving complete proofs whenever such modifications appear to change the basic theory of diffeological connections (they don't).

\subsection{Basics of Diffeology}
\label{subsec:BasicsOfDiffeology}

Let $X$ be any set (we do not assume a priori any structure or topology on $X$). The following material can be found in \citet[Ch.~1]{PIZ}.
\begin{definition}[{\citealp[art.~1.5]{PIZ}}]
    \label{def:Diffeology}
    A \textbf{diffeology} on $X$ is a collection $\mathcal D$ of maps $\psi : U \to X$, where $U \se \R^n$ is an open subset (here $n \geq 0$ is allowed to vary), subject to the following axioms:
    \begin{enumerate}
        \item[(i)] \emph{Covering}: If $\psi : U \to X$ is constant, then $\psi \in \mathcal D$.
        \item[(ii)] \emph{Locality}: If, given $\psi : U \to X$, there exists an open cover $\bigcup_i U_i = U$ such that each $\psi\upharpoonright U_i \in \mathcal D$, then $\psi \in \mathcal D$.
        \item[(iii)] \emph{Chain Rule}: If $\psi : \R^n \supseteq U \to X$ is a member of $\mathcal D$ and $F : \R^k \supseteq V \to U$ is a $\smooth$-map, then $\psi \circ F \in \mathcal D$.
    \end{enumerate}
    We call the pair $(X, \mathcal D)$ a \textbf{diffeological space}. Members of $\mathcal D$ are called \textbf{plots for $X$}.
\end{definition}

\begin{definition}[{\citealp[art.~1.14]{PIZ}}]
    \label{def:SmoothMap}
    A map $f : X \to Y$ of diffeological spaces is called \textbf{smooth} if for each plot $\psi$ for $X$, the composite $f\circ \psi$ is a plot for $Y$. The set of smooth maps $X \to Y$ is denoted by $\smooth(X, Y)$. A \textbf{diffeomorphism} is a smooth bijection whose inverse is smooth. The set of diffeomorphisms $X \to Y$ is denoted by $\diff(X, Y)$, or briefly $\diff(X)$ in the case that $Y = X$.
\end{definition}

\begin{definition}[Basic Constructions in Diffeology]
    \label{def:BasicConstructionsDiffeology}
    Let $X$ and $Y$ be diffeological spaces, and let $\{ X_j : j \in J \}$ be an arbitrary family of diffeological spaces.
    \begin{enumerate}
        \item[(i)] \textbf{Subspace Diffeology} \citep[art.~1.33]{PIZ}. A subset $A \se X$ is a diffeological space whose plots are the maps $U \to A$ which are plots for $X$ when regarded as maps $U \to X$.
        \item[(ii)] \textbf{Product Diffeology} \citep[art.~1.55]{PIZ}. The cartesian product $\prod_{j \in J} X_j$ is a diffeological space whose plots are the maps $\psi : U \to \prod_{j \in J} X_j$, $\psi(r) = (\psi_j(r))_{j \in J}$, such that each component $\psi_j : U \to X_j$ is a plot.
        \item[(iii)] \textbf{Quotient Diffeology} \citep[art.~1.50]{PIZ}. For an equivalence relation $\sim$ on $X$, the quotient $X/\!\sim$ is a diffeological space whose plots are the maps $\psi : U \to X/\!\sim$ which locally lift to plots for $X$. More explicitly: for each $r_0 \in U$, there exists a neighborhood $V \se U$ of $r_0$ and a plot $r \mapsto x_r$ for $X$ defined on $V$ such that
        \[
            (\psi\upharpoonright V)(r) = [x_r],
        \]
        where $[\cdot]$ denotes the equivalence class for $\sim$.
        \item[(iv)] \textbf{Functional Diffeology} \citep[art.~1.57]{PIZ}. The set of smooth maps $\smooth(X,Y)$ is a diffeological space whose plots are the maps
        \[
            \psi : U \to \smooth(X,Y),\quad r \mapsto \psi_r
        \]
        such that the induced evaluation map
        \[
            \hat \psi : U \times X \to Y,\quad  \hat \psi(r,x) = \psi_r(x)
        \]
        is smooth.
    \end{enumerate}
\end{definition}

\begin{lemma}[{\citealp[art.~1.59]{PIZ}}]
    \label{lem:CompositionIsSmooth}
    Let $X, Y$, and $Z$ be diffeological spaces. Then, the composition map
    \begin{align*}
        \circ : \smooth(X,Y) \times \smooth(Y, Z) &\to \smooth(X, Z) \\
        (f, g) &\mapsto g \circ f
    \end{align*}
    is smooth.
\end{lemma}

The fact that the set of smooth maps between two diffeological spaces is again a diffeological space is one of the main reasons that diffeology is well-suited to the study of path spaces. We will need the following variant, which allows the domains of the maps to vary:
\begin{definition}[{\citealp[art.~1.63]{PIZ}}]
    \label{def:FunctionalDiffeologyOfDiffeology}
    Let $\mathcal D$ be a diffeology on $X$. Then $\mathcal D$ itself carries a natural diffeology called the \textbf{functional diffeology of a diffeology}. Namely, a parameterization $\psi : U \to \mathcal D$, $r \mapsto \psi_r$ for $\mathcal D$ is a plot if and only if for each $r_0 \in U$ and $s_0 \in \dom(\psi_{r_0})$, there are open neighborhoods $V \se U$ of $r_0$ and $W$ of $s_0$ such that:
    \begin{enumerate}[(i)]
        \item For all $r \in V$, one has $W \se \dom(\psi_r)$.
        \item The evaluation map $V \times W \to X$ given by $(r,s) \mapsto \psi_r(s)$ is a plot for $X$.
    \end{enumerate}
\end{definition}
The functional diffeology of a diffeology is used to define path spaces in Section \ref{subsec:ConnectionsAndHolonomyInDiffeology}.

\subsection{Diffeological Fiber and Principal Bundles}
\label{subsec:BundlesInDiffeology}

Diffeological fiber bundles are discussed in \citet[Ch.~8]{PIZ}. From here on, our conventions and notation begin to deviate from those of the book; we indicate where.
\begin{definition}[{\citealp[art.~8.9]{PIZ}}]
    \label{def:DiffeologicalFiberBundle}
    Let $E, F,$ and $X$ be diffeological spaces. A \textbf{diffeological fiber bundle with fiber $F$} is a smooth surjection $\pi : E \to X$ such that, for each plot $\psi : U \to X$, the pullback
    \[
        \psi^*E = \{ (r, p) \in U \times E : \psi(r) = \pi(p) \},
    \]
    equipped with the subspace diffeology of $U \times E$, is locally trivial over $U$ with fiber $F$:
    \[
    \begin{tikzcd}
        \psi^*E \arrow{r}{\pr_2} \arrow[swap]{d}{\pr_1} \arrow[dr, phantom, "\lrcorner", very near start] & E \arrow{d}{\pi} \\
        U \arrow[swap]{r}{\psi} & X
    \end{tikzcd}
    \]
    That is, each point $r \in U$ has a neighborhood $V \se U$ for which there is a diffeomorphism $T : \pr_1\inverse(V) \to V \times F$ with $\pr_1 \circ T = \pr_1$. Such a $T$ is called a \textbf{plot trivialization} of $\pi : E \to X$.
\end{definition}
Iglesias-Zemmour's definition of a diffeological fibration is first phrased in terms of its structure groupoid \citep[art.~8.8]{PIZ}. By \citet[art.~8.9]{PIZ}, it is equivalent to the local triviality condition described in Definition \ref{def:DiffeologicalFiberBundle}, which we take as our definition.

Let $G$ be a diffeological group, i.e. a group equipped with a diffeology for which multiplication and inversion are smooth \citep[art.~7.1]{PIZ}. As is standard in gauge theory, we take $G$ to act from the \emph{right}: a \textbf{smooth right action} of $G$ on a diffeological space $E$ is a smooth map $E \times G \to E$, $(\xi, g) \mapsto \xi \cdot g$, such that $\xi \cdot e = \xi$ and $(\xi \cdot g) \cdot h = \xi \cdot (gh)$ for all $\xi \in E$ and $g, h \in G$. When there is no confusion, we write $\xi g$ for $\xi \cdot g$. Iglesias-Zemmour \citeyearpar[arts.~7.4 and 8.11]{PIZ} works instead with left actions, given by smooth homomorphisms $G \to \diff(E)$. Since a right action always corresponds to the left action $g \mapsto (\xi \mapsto \xi \cdot g\inverse)$, we translate the results of the book to right actions without further comment.

Given a smooth right action of $G$ on $E$, the map
\[
\Gamma : E \times G \to E \times E, \quad \Gamma(\xi, g) = (\xi, \xi g)
\]
is smooth. Iglesias-Zemmour \citeyearpar[art.~8.11]{PIZ} showed that if $\Gamma$ is a diffeomorphism onto its image (equipped with the subspace diffeology of $E \times E$), then the quotient map $E \to E/G$ is a diffeological fiber bundle with fiber $G$. Such a quotient map is what he calls a \emph{principal fiber bundle}. We will use the following form of this definition:
\begin{definition}[{\citealp[art.~8.11]{PIZ}}]
    \label{def:DiffeologicalPrincipalBundle}
    Let $\pi : E \to X$ be a smooth surjection of diffeological spaces and $G$ a diffeological group which acts on $E$ smoothly from the right. Then $\pi : E \to X$ is a \textbf{diffeological principal $G$-bundle} if the following hold:
    \begin{enumerate}[(i)]
        \item The map $\Gamma : E \times G \to E \times E$, $\Gamma(\xi, g) = (\xi, \xi g)$, is a diffeomorphism onto its image.
        \item There is a diffeomorphism $\Phi : X \to E/G$ making the following triangle commute:
        \[
        \begin{tikzcd}[column sep=small]
            & E \arrow[swap]{dl}{\pi} \arrow{dr}{[\cdot]} & \\
            X \arrow[rr, "\Phi", "\cong"'] & & E/G
        \end{tikzcd}
        \]
        That is, $\Phi(\pi(\xi)) = [\xi] = \{ \xi g : g \in G \}$ for all $\xi \in E$.
    \end{enumerate}
\end{definition}

Condition (i) of Definition \ref{def:DiffeologicalPrincipalBundle} implies that the action of $G$ on $E$ is free. Note also that the image of $\Gamma$ is precisely the fiber product $\{ (p, q) \in E \times E : \pi(p) = \pi(q) \}$. The following equivalent characterization of diffeological principal $G$-bundles is more oriented towards the local trivialization point of view.
\begin{proposition}[{\citealp[art.~8.13]{PIZ}}]
    \label{prop:GEquivariantTrivializations}
    Let $\pi : E \to X$ be a smooth surjection and $G$ a diffeological group acting smoothly on $E$ from the right such that $\pi(\xi g) = \pi(\xi)$ for all $\xi \in E$ and $g \in G$. The following are equivalent:
    \begin{enumerate}
        \item[(i)] $\pi : E \to X$ is a diffeological principal $G$-bundle.
        \item[(ii)] For each plot $\psi : U \to X$, there is an open cover $\bigcup_i U_i = U$ and plots $\Psi_i : U_i \to E$ with $\pi \circ \Psi_i = \psi \upharpoonright U_i$, such that the maps
        \[
            T_i : U_i \times G \to \pr_1\inverse(U_i) \se \psi^*E, \quad T_i(r, g) = (r, \Psi_i(r)g),
        \]
        are diffeomorphisms.
    \end{enumerate}
    Note that $T_i(r, g_1 g_2) = T_i(r, g_1) g_2$ for all $g_1, g_2 \in G$ with the evident $G$-action on the right-hand-side. We call the $T_i$ \textbf{$G$-equivariant plot trivializations}.
\end{proposition}

\begin{lemma}
    \label{lem:CoveringPlotsDifferByGplot}
    Let $\pi : E \to X$ be a diffeological principal $G$-bundle. If $\Psi : U \to E$ and $\Psi' : U \to E$ are plots for $E$ such that $\pi \circ \Psi = \pi \circ \Psi'$, then there exists a unique plot $U \to G$, $r \mapsto g_r$, such that $\Psi'(r) = \Psi(r)g_r$ for all $r \in U$.
\end{lemma}
\begin{proof}
    By freeness, we at least have a unique map $r \mapsto g_r$ with $\Psi'(r) = \Psi(r) g_r$. Now $r \mapsto (\Psi(r), \Psi'(r))$ is a plot for $E \times E$ with values in the image of $\Gamma$, and $\Gamma$ is a diffeomorphism onto its image, so $r \mapsto \Gamma\inverse(\Psi(r), \Psi'(r)) = (\Psi(r), g_r)$ is a plot for $E \times G$. Hence $r \mapsto g_r$ is a plot for $G$.

\end{proof}

A necessary construction for us is that of the \textit{associated bundle}. Associated bundles are discussed in \citet[art.~8.16]{PIZ}, but we need a slightly different presentation, which we discuss now.

Let $\pi : E \to X$ be a diffeological principal $G$-bundle and $\varphi : G \to H$ a smooth group homomorphism into a diffeological group $H$. Then there is a smooth right action of $G$ on the product space $E \times H$ given by
\[
(\xi, h)g = (\xi g, \varphi(g)\inverse h),
\]
where the first component is the action of $G$ on $E$ and the second component is the group multiplication in $H$. Denote by $E \times_\varphi H$ the quotient of $E \times H$ by this action, i.e. elements of $E \times_\varphi H$ are equivalence classes denoted by
\[
\llbracket \xi, h\rrbracket_\varphi = \{ (\xi g, \varphi(g)\inverse h) : g \in G \}.
\]
When the group homomorphism $\varphi : G \to H$ is clear, we may drop the subscript $\varphi$ and write $\llbracket \xi, h \rrbracket$. Being a quotient of a product of diffeological spaces, $E \times_\varphi H$ is naturally a diffeological space.

\begin{definition}
    \label{def:AssociatedBundle}
    Let $\pi : E \to X$ be a diffeological principal $G$-bundle and $\varphi : G \to H$ a smooth group homomorphism into a diffeological group $H$. The \textbf{associated bundle} of $\pi : E \to X$ and $\varphi : G \to H$ is the well-defined map $\pi_\varphi : E \times_\varphi H \to X$ given by
    \[
    \pi_\varphi(\llbracket \xi, h\rrbracket) = \pi(\xi),
    \]
    which is a diffeological principal $H$-bundle with the $H$-action given by
    \[
        \llbracket \xi, h \rrbracket h' = \llbracket \xi, hh' \rrbracket.
    \]
\end{definition}

The treatment of associated bundles in \citet[art.~8.16]{PIZ}, including a proof that $E\times_\varphi H$ is in fact a diffeological principal bundle, replaces $H$ by a general diffeological space $Y$ (not necessarily a group) and the homomorphism $\varphi$ by a smooth left action of $G$ on $Y$, and shows that the associated bundle is a diffeological fiber bundle with fiber $Y$. Our $E \times_\varphi H$ is the special case where $Y = H$ and $G$ acts on $Y$ by $G \to \diff(Y)$ is the composition $G \ni g \mapsto \varphi(g) \mapsto \varphi(g) (\cdot) \in \mathrm{Diff}(Y)$.

\subsection{Diffeological Connections and Holonomy Representations}
\label{subsec:ConnectionsAndHolonomyInDiffeology}

The central object of our study is the \textit{connection}. A notion of a \emph{diffeological connection} on diffeological principal $G$-bundles was introduced in \citet[art.~8.32]{PIZ}, and its holonomy is discussed in \citet[art.~8.35]{PIZ}. In the paragraph just before art.~8.32, Iglesias-Zemmour says that the concept of a connection in diffeology is ``more a program than a strict immutable definition.'' We will take the liberty then to modify Iglesias-Zemmour's definition of a connection slightly (see Remarks \ref{rem:DifferentConnectionDefinition} and \ref{rem:ProjectionAxiom}), while also using notation more tailored to the needs of this paper. Results drawn from the book that appear as if they would be changed by our modified definition do not, and we include proofs in these instances.

Let us first establish some notation. For a diffeological space $(Y, \mathcal D)$, denote by $\P Y$ the subset of $\mathcal D$ consisting of plots $\alpha : I \to Y$, where $I \se \R$ is an open interval (possibly unbounded). Also denote by $\Pin Y$ the set
\[
    \Pin Y = \{ (\alpha, t) \in \P Y \times \R : t \in \dom(\alpha) \}.
\]
We endow $\P Y \se \mathcal D$ and $\Pin Y \se \P Y \times \R$ with the subspace diffeologies, where $\mathcal D$ carries the functional diffeology of a diffeology (Definition \ref{def:FunctionalDiffeologyOfDiffeology}). If $f : Y \to Z$ is smooth, then post-composition $\P Y \to \P Z$, $\alpha \mapsto f \circ \alpha$ and  evaluation $\Pin Y \to Y$, $(\alpha, t) \mapsto \alpha(t)$ are smooth maps.

\begin{definition}
    \label{def:PathSpaces}
    Let $Y$ be a diffeological space.
    \begin{enumerate}[(i)]
        \item The diffeological space $\P Y$ is called the \textbf{path space} of $Y$. Members $\alpha$ of $\P Y$ are called \textbf{smooth paths}.
        \item The diffeological space $\Pin Y$ is called the \textbf{initialized path space} of $Y$. Members $(\alpha, t)$ of $\Pin Y$ are called \textbf{initialized smooth paths}.
        \item A \textbf{smooth loop} is an element $\gamma \in \P Y$ such that $0,1 \in \dom(\gamma)$ and $\gamma(0) = \gamma(1)$. If $y_0 = \gamma(0) = \gamma(1)$, we say that $\gamma$ is a \textbf{smooth loop based at $y_0$.}
        \item The diffeological subspace of $\P Y$ consisting of smooth loops based at $y_0$ is denoted by $\mathcal L(Y, y_0)$.
    \end{enumerate}
\end{definition}

Given a bundle $\pi : E \to X$, we will usually denote paths in the total space $E$ with an overhead tilde, such as $\tilde \alpha$ or $\tilde \beta$. Paths in the base space $X$ will usually come with no decorations, such as $\alpha$ or $\beta$.

\begin{definition}[{cf. \citealp[art.~8.32]{PIZ}}]
    \label{def:Connection}
    Let $\pi : E \to X$ be a diffeological principal $G$-bundle. By a \textbf{(diffeological) connection} on $\pi : E \to X$, we mean a smooth map
    \[
        A : \Pin E \to \P E, \quad (\tilde \alpha, t) \mapsto A_t\tilde \alpha
    \]
    satisfying the following conditions for every $(\tilde \alpha, t) \in \Pin E$:
    \begin{enumerate}[(i)]
        \item \emph{Domain}: $\dom(A_t\tilde \alpha) = \dom(\tilde \alpha).$
        \item \emph{Lifting}: $\pi \circ A_t\tilde \alpha = \pi \circ \tilde \alpha$.
        \item \emph{Basepoint}: $A_t\tilde \alpha(t) = \tilde \alpha(t).$
        \item \emph{$G$-equivariance}: $A_t(\tilde \alpha \cdot \rho) = (A_t\tilde \alpha) \cdot \rho(t)$ for every smooth path $\rho : \dom(\tilde \alpha) \to G$. Here $\tilde \alpha \cdot \rho$ is the path $s \mapsto \tilde \alpha(s)\cdot \rho(s)$, in which each point is acted on by a (possibly) different group element determined by $\rho$, whereas $(A_t\tilde \alpha) \cdot \rho(t)$ is the path $s \mapsto (A_t\tilde \alpha)(s) \cdot \rho(t)$, in which every point is acted on by the fixed group element $\rho(t)$.
        \item \emph{Reparameterization}: $A_s(\tilde \alpha \circ f) = (A_{f(s)}\tilde \alpha) \circ f$ for every piecewise, weakly monotone $f \in \mathcal P(\dom(\tilde \alpha))$ and every $s \in \dom(f)$. By \textit{piecewise, weakly monotone}, we mean that there is a finite partition $t_1 < t_2 < \cdots < t_n$ of the domain of $f$ for which $f$ is weakly monotone on each segment.
        \item \emph{Projection}: $A_s(A_t\tilde \alpha) = A_t\tilde \alpha$ for all $s \in \dom(\tilde \alpha)$.
    \end{enumerate}
    We call the path $A_t\tilde \alpha$ the \textbf{horizontal projection of $\tilde \alpha$ at time $t$}.
\end{definition}

\begin{remark}
    \label{rem:DifferentConnectionDefinition}
    The above Definition \ref{def:Connection} is nearly identical to \citet[art.~8.32]{PIZ}, with just the following differences:
    \begin{enumerate}[(a)]
        \item We take actions by $G$ from the \textit{right}. Admittedly, this is not a real difference.
        \item Our reparameterization axiom (Definition \ref{def:Connection}(v)) by piecewise, weakly monotone $f$ is weaker than Iglesias-Zemmour's reparameterization axiom, who takes reparameterization by \textit{arbitrary} $f \in \P \dom(\tilde \alpha)$.
        \item Our projection axiom (Definition \ref{def:Connection}(vi)) is stronger than Iglesias-Zemmour's projection axiom, who takes $s = t$ instead.
    \end{enumerate}
    It is important to reiterate that the basic theory of diffeological connections developed in \citet[arts.~8.32--8.35]{PIZ} remain unchanged after switching to our Definition \ref{def:Connection}, as we show if and when the situation arises. For example, the arguments there which use the reparameterization axiom happen to use it for piecewise, weakly monotone $f$ anyway. The reason for the stronger projection axiom is explained in Remark \ref{rem:ProjectionAxiom} below.
\end{remark}

\begin{remark}
    \label{rem:ProjectionAxiom}
    The first draft of this manuscript used Iglesias-Zemmour's projection axiom, i.e. that $A_t(A_t\tilde \alpha) = A_t \tilde \alpha$. The anonymous referee raised the point that a genuine ``projection'' should really satisfy $A_s(A_t\tilde \alpha) = A_t \tilde \alpha$ for all $s,t \in \dom(\tilde \alpha)$ and asked if this stronger axiom follows from the previous ones. The author thanks the referee for this question. It turns out that axiom (vi) does not follow from axioms (i)-(v) together with the axiom $A_t(A_t\tilde \alpha) = A_t \tilde \alpha$, even if (v) is replaced by Iglesias-Zemmour's stronger reparameterization axiom. Moreover, without (vi), horizontal paths (Definition \ref{def:HorizontalPath}) which cover the same base path and agree at some point may not coincide. 
    
    To see this, let $\pi : E = \R \times \R \to \R = X$, $\pi(x,y) = x$, be the trivial principal $G$-bundle with $G = (\R, +)$ acting by $(x, y) \cdot g = (x, y + g)$. For a path $\tilde \alpha(s) = (x(s), y(s))$ in $E$ and an initial time $t \in \dom(\tilde \alpha)$, put
    \[
        A_t\tilde \alpha(s) = \big(x(s),\ y(t) + (x(s) - x(t))^2\big).
    \]
    One checks directly that $A$ is smooth, satisfies (i)--(iv), satisfies (v) for every smooth $f$, and satisfies $A_t(A_t\tilde \alpha) = A_t\tilde \alpha$. However,
    \[
        A_u\big(A_t\tilde \alpha\big)(s) = \big( x(s),\ y(t) + (x(s) - x(t))^2 + 2(x(s) - x(u))(x(t) - x(u)) \big),
    \]
    so (vi) fails in general. For the last point, note that for $\tilde \alpha(s) = (s, 0)$, the horizontal paths $A_0\tilde \alpha(s) = (s, s^2)$ and $A_2\tilde \alpha(s) = (s, (s-2)^2)$ cover the same path in the base $X$ and agree at $s = 1$, yet they are different.
\end{remark}

\begin{definition}
    \label{def:HorizontalPath}
    Let $\pi : E \to X$ be a diffeological principal $G$-bundle equipped with a connection $A : \Pin E \to \P E$. A path $\bar \alpha \in \P E$ is called \textbf{$A$-horizontal} if $A_t \bar \alpha = \bar \alpha$ for some $t \in \dom(\bar \alpha)$. We may also say $\bar \alpha$ is simply \textbf{horizontal} if the connection is clear. Horizontal paths will usually be denoted with an overhead bar, i.e. $\bar \alpha$ or $\bar \beta$.
\end{definition}

\begin{remark}
    \label{rem:WitnessingHorizontality}
    If $\bar \alpha$ is horizontal, then this horizontality is witnessed by some time $t$ according to the above Definition \ref{def:HorizontalPath}; write $\bar \alpha = A_t \bar \alpha$. Then by axiom (vi) of Definition \ref{def:Connection}, one has that $A_s \bar \alpha = A_s(A_t\bar \alpha) = A_t \bar \alpha = \bar \alpha$ for all $s \in \dom(\bar \alpha)$. In other words, if a horizontal path is witnessed by some time $t$, then it is witnessed by any and all times $s$. Moreover, by axiom (vi), any path of the form $A_t\tilde \alpha$ for an arbitrary $\tilde \alpha \in \P E$ is $A$-horizontal.
\end{remark}

\begin{proposition}[{cf. \citealp[art.~8.32(a)]{PIZ}}]
    \label{prop:HorizontalPathsThatMeetCoincide}
    Let $\pi : E \to X$ be a diffeological principal $G$-bundle equipped with a connection $A : \Pin E \to \P E$. If $\bar \alpha, \bar \beta : I \to E$ are $A$-horizontal paths such that $\pi \circ \bar \alpha = \pi \circ \bar \beta$, then $\bar \alpha = \bar \beta$ if and only if $\bar \alpha$ and $\bar \beta$ coincide at least at one point.
\end{proposition}
\begin{proof}
    Take $s_0 \in I$ such that $\bar \alpha(s_0) = \bar \beta(s_0)$. By Remark \ref{rem:WitnessingHorizontality}, we may write for this common time $s_0$ that $\bar \alpha = A_{s_0} \bar \alpha$ and $\bar \beta = A_{s_0} \bar \beta$. By Lemma \ref{lem:CoveringPlotsDifferByGplot}, there is a smooth path $\rho : I \to G$ such that $\bar \beta(s) = \bar \alpha(s) \cdot \rho(s)$ for all $s \in I$. Freeness implies that $\rho(s_0) = e$. In particular, we have
    \[
    \bar \beta(s) = A_{s_0} \bar \beta(s) = A_{s_0}(\bar \alpha(s) \cdot \rho(s)) = (A_{s_0}\bar \alpha(s)) \cdot \rho(s_0) = A_{s_0} \bar \alpha(s) = \bar \alpha(s),
    \]
    where we have used $G$-equivariance (axiom (iv)) of $A$.

\end{proof}

Traditionally, connections may be defined as a certain \emph{horizontal lifting function} which satisfies a variety of conditions. A diffeological connection gives rise to such a lift. Let
\[
    \ev_X^*(E) = \{ (\alpha, t_0, \xi_0) \in \Pin X \times E : \pi(\xi_0) = \alpha(t_0) \},
\]
be equipped with the subspace diffeology.

\begin{proposition}[{\citealp[art.~8.32(e)]{PIZ}}]
    \label{prop:HorizontalLiftingFunction}
    Let $\pi : E \to X$ be a diffeological principal $G$-bundle equipped with a connection $A : \Pin E \to \P E$. For every $(\alpha, t_0, \xi_0) \in \ev_X^*(E)$, there is a unique $A$-horizontal path $\hor_A(\alpha, t_0, \xi_0) \in \P E$ such that
    \begin{enumerate}[(i)]
        \item $\pi \circ \hor_A(\alpha, t_0, \xi_0) = \alpha$.
        \item $\hor_A(\alpha, t_0, \xi_0)(t_0) = \xi_0.$
    \end{enumerate}
    Moreover, 
    \[
    A_t \tilde \alpha = \hor_A(\pi \circ \tilde \alpha, t, \tilde \alpha(t))
    \]
    for every $(\tilde \alpha, t) \in \Pin E$. We call $\hor_A(\alpha, t_0, \xi_0)$ the \textbf{horizontal lift of $\alpha$ through $\xi_0$ at time $t_0$}, and $\hor_A : \ev_X^*(E) \to \mathcal P E$ the \textbf{horizontal lifting function} of $A$. It is smooth and $G$-equivariant.
\end{proposition}
\begin{proof}
    Everything except for the equation $A_t\tilde \alpha = \hor_A(\pi\circ \tilde \alpha, t, \tilde \alpha(t))$ is shown in \citet[art.~8.32(e)]{PIZ}. Since the basepoint property (Definition \ref{def:Connection}(iii)) of $A$ implies $A_t\tilde \alpha(t) = \tilde \alpha(t) = \hor_A(\tilde \alpha, t, \tilde \alpha(t))(t)$, this equation follows simply by Proposition \ref{prop:HorizontalPathsThatMeetCoincide}. 

\end{proof}

\begin{remark}
    The reason for the notation $\ev_X^*(E)$ is that $\ev_X^*(E)$ is the pullback of $\pi : E \to X$ along the evaluation map $\ev_X : \Pin X \to X$ given by $\ev_X(\alpha, t) = \alpha(t)$.
\end{remark}

Later, we will consider paths which are concatenations of other paths. The formal treatment of path concatenation is delayed to Section \ref{sec:PathSpaces}, but we provide a preliminary definition for paths defined on open intervals:

\begin{definition}
    \label{def:OpenPathConcatenation}
    Let $Y$ be a diffeological space.
    \begin{enumerate}[(1)]
        \item Let $\alpha, \beta \in \P Y$ with $1 \in \dom(\alpha)$, $0 \in \dom(\beta)$, and $\beta(0) = \alpha(1)$. Define $\beta \vee \alpha : I \to Y$, where $I$ is an appropriate interval, by
        \[
            (\beta \vee \alpha)(t) =
            \begin{dcases}
                \alpha(2t) & \text{if } t \leq \frac{1}{2}, \\
                \beta(2t - 1) & \text{if } t \geq \frac{1}{2}.
            \end{dcases}
        \]
        If $\beta \vee \alpha$ is smooth, i.e. $\beta \vee \alpha \in \mathcal P Y$, then it is called the \textbf{concatenation of $\alpha$ and $\beta$}, read ``$\beta$ after $\alpha$.''\footnote{Iglesias-Zemmour \citeyearpar[art.~5.2]{PIZ} writes $c \vee c'$ for the concatenation which runs through $c$ first; our order follows the convention for composition, which will play nicely with right-actions later in Section \ref{subsec:UniversalConnection}.} This is the case, for instance, if $\alpha$ is constant near $1$ and $\beta$ is constant near $0$.
        \item The \textbf{reversal} of $\alpha \in \P Y$ is the path $\alpha^\leftarrow \in \P Y$ defined by $\alpha^\leftarrow(t) = \alpha(1 - t)$ on $1 - \dom(\alpha)$.
    \end{enumerate}
\end{definition}

\begin{proposition}[{cf. \citealp[art.~8.32(e,f)]{PIZ}}]
    \label{prop:PropertiesOfHorizontalLift}
    Let $\pi : E \to X$ be a diffeological principal $G$-bundle equipped with a connection $A : \Pin E \to \P E$, and let $(\alpha, t_0, \xi_0) \in \ev_X^*(E)$. The horizontal lifting function $\hor_A : \ev_X^*(E) \to \P E$ has the following properties:
    \begin{enumerate}[(i)]
        \item \emph{$G$-equivariance}: $\hor_A(\alpha, t_0, \xi_0 g) = \hor_A(\alpha, t_0, \xi_0)\cdot g$ for all $g \in G$.
        \item \emph{Reparameterization}: $\hor_A(\alpha \circ f, s_0, \xi_0) = \hor_A(\alpha, f(s_0), \xi_0) \circ f$ for every piecewise monotone $f \in \P(\dom(\alpha))$ and every $s_0 \in \dom(f)$.
        \item \emph{Concatenation}: Let $\alpha, \beta \in \P X$ be such that $0, 1 \in \dom(\alpha)$, $0 \in \dom(\beta)$, and $\beta \vee \alpha \in \P X$. Then, for every $\xi_0 \in \pi\inverse(\alpha(0))$,
        \[
            \hor_A(\beta \vee \alpha, 0, \xi_0) = \hor_A(\beta, 0, \xi_1) \vee \hor_A(\alpha, 0, \xi_0), \quad \text{where } \xi_1 = \hor_A(\alpha, 0, \xi_0)(1).
        \]
        \item \emph{Reversal}: Let $\alpha \in \P X$ be such that $0, 1 \in \dom(\alpha)$. Then, for every $\xi_0 \in \pi\inverse(\alpha(0))$,
        \[
            \hor_A(\alpha^\leftarrow, 0, \xi_1) = \hor_A(\alpha, 0, \xi_0)^\leftarrow, \quad \text{where } \xi_1 = \hor_A(\alpha, 0, \xi_0)(1).
        \]
        In particular, $\hor_A(\alpha^\leftarrow, 0, \xi_1)(1) = \xi_0$. 
    \end{enumerate}
\end{proposition}
\begin{proof}
    Items (i)--(iii) are in \citet[art.~8.32]{PIZ}. For (iv), let $\bar \alpha_{\xi_0} = \hor_A(\alpha, 0, \xi_0)$. Then, for the (piecewise, weakly) monotone $f(t) = 1 - t$, we have that
    \begin{align*}
        \hor_A(\alpha^\leftarrow, 0, \xi_1) &= \hor_A(\alpha \circ f, 0, \xi_1) \\
        &= \hor_A(\alpha, 1, \xi_1) \circ f \\
        &= \hor_A(\pi \circ \bar \alpha_{\xi_0}, 1, \bar \alpha_{\xi_0}(1)) \circ f \\
        &= (A_1 \bar \alpha_{\xi_0}) \circ f \\
        &= \bar \alpha_{\xi_0} \circ f \\
        &= \hor_A(\alpha, 0, \xi_0)^\leftarrow.
    \end{align*}
    where we have used the fact that $A_t \tilde \alpha = \hor_A(\pi \circ \tilde \alpha, t, \tilde \alpha(t))$ from Proposition \ref{prop:HorizontalLiftingFunction}.

\end{proof}

Suppose now that $\pi : (E, \xi_0) \to (X, x_0)$ is a pointed diffeological principal $G$-bundle, i.e. $\pi(\xi_0) = x_0$, equipped with a connection $A : \Pin E \to \P E$. Given a loop $\gamma \in \mathcal L(X, x_0)$, its horizontal lift $\hor_A(\gamma, 0, \xi_0)$ need \emph{not} be a loop in $E$, although its endpoint $\hor_A(\gamma, 0, \xi_0)(1)$ lies in the fiber $\pi\inverse(x_0)$. This means that there is a unique element $g_{\gamma, \xi_0}^A \in G$ for which $\hor_A(\gamma, 0, \xi_0)(1) = \xi_0\cdot g_{\gamma, \xi_0}^A$, which gets a special name:
\begin{definition}[{cf. \citealp[art.~8.35]{PIZ}}]
    \label{def:HolonomyRepresentation}
    Let $\pi : (E, \xi_0) \to (X, x_0)$ be a pointed diffeological principal $G$-bundle. A diffeological connection $A : \Pin E \to \P E$ gives rise to a mapping
    \[
        H^A_{\xi_0} : \mathcal L(X, x_0) \to G, \quad H^A_{\xi_0}(\gamma) = g^A_{\gamma, \xi_0},
    \]
    where $g^A_{\gamma, \xi_0}$ is as above, which is called the \textbf{holonomy representation of $A$ at the point $\xi_0$}. The \textbf{holonomy group of $A$ based at $\xi_0$} is the subgroup $\Hol_{\xi_0}(A) \leq G$ defined by
    \[
        \Hol_{\xi_0}(A) = \left\{g_{\gamma, \xi_0}^A \in G : \gamma \in \mathcal L(X, x_0) \right\} = H^A_{\xi_0}\big(\mathcal L(X, x_0)\big).
    \]
\end{definition}
\begin{proposition}
    \label{prop:HolonomyRepresentationIsSmooth}
    Let $\pi : (E, \xi_0) \to (X, x_0)$ be a pointed diffeological principal $G$-bundle equipped with a connection $A$. Then the holonomy representation $H^A_{\xi_0} : \mathcal L(X, x_0) \to G$ is smooth.
\end{proposition}
\begin{proof}
    The orbit map $o_{\xi_0} : G \to \pi\inverse(x_0)$, $o_{\xi_0}(g) = \xi_0 g$, is a diffeomorphism onto the fiber $\pi\inverse(x_0)$ because $\Gamma : E \times G \to E \times E$ is a diffeomorphism onto its image. In particular, its inverse $o_{\xi_0}\inverse$ is smooth. Hence $H^A_{\xi_0}$, being the composite
    \[
    \begin{tikzcd}[every label/.append style = {yshift=3pt}]
        \mathcal L(X, x_0) \arrow{r}{\hor_A(\cdot, 0, \xi_0)} & \mathcal P_1 E \arrow{r}{\ev_1} & \pi\inverse(x_0) \arrow{r}{o_{\xi_0}\inverse} & G,
    \end{tikzcd}
    \]
    where $\mathcal P_1 E \se \mathcal P E$ is the subspace of paths whose domain contains $1$ and $\ev_1(\bar \alpha) = \bar \alpha(1)$, is smooth.

\end{proof}

\begin{lemma}
    \label{lem:AlmostGroupHomomorphism}
    Let $\pi : (E, \xi_0) \to (X, x_0)$ be a pointed diffeological principal $G$-bundle equipped with a connection $A$, and let $\gamma, \sigma \in \mathcal L(X, x_0)$ be such that $\gamma \vee \sigma \in \mathcal L(X, x_0)$. Then $H^A_{\xi_0}(\gamma \vee \sigma) = H^A_{\xi_0}(\gamma) \cdot H^A_{\xi_0}(\sigma)$.
\end{lemma}
\begin{proof}
    Write $g = H^A_{\xi_0}(\sigma)$ and $\xi_1 = \xi_0 \cdot g = \hor_A(\sigma, 0, \xi_0)(1)$. By $G$-equivariance of $\hor_A$ (Proposition \ref{prop:PropertiesOfHorizontalLift}(i)), we have
    \[
        \xi_0  \cdot g H^A_{\xi_1}(\gamma) = \xi_1 \cdot H^A_{\xi_1}(\gamma) = \hor_A(\gamma, 0, \xi_1)(1) = \hor_A(\gamma, 0, \xi_0)(1) \cdot g = \xi_0 \cdot H^A_{\xi_0}(\gamma) g
    \]
    so, in particular, $H^A_{\xi_1}(\gamma) = g\inverse H^A_{\xi_0}(\gamma) g$. By the concatenation property of $\hor_A$ (Proposition \ref{prop:PropertiesOfHorizontalLift}(iii)), one has
    \begin{align*}
        \xi_0 \cdot H^A_{\xi_0}(\gamma \vee \sigma) &= \hor_A(\gamma \vee \sigma, 0, \xi_0)(1) \\
        &= \hor_A(\gamma, 0, \xi_1)(1) \\
        &= \xi_1 \cdot H^A_{\xi_1}(\gamma) \\
        &= \xi_0 \cdot H^A_{\xi_0}(\gamma) H^A_{\xi_0}(\sigma),
    \end{align*}
    so that $H^A_{\xi_0}(\gamma \vee \sigma) = H^A_{\xi_0}(\gamma) \cdot H^A_{\xi_0}(\sigma)$.

\end{proof}

Lemma \ref{lem:AlmostGroupHomomorphism} makes $H^A_{\xi_0} : \mathcal L(X, x_0) \to G$ look like a group homomorphism except for the fact that $\mathcal L(X, x_0)$ is not a group. In our treatment of paths in Section \ref{subsec:PathConcatenation}, we will see how $H^A_{\xi_0}$ yields a smooth homomorphism of diffeological groups.

\section{Diffeological Path Spaces and Singer's Universal Connection}
\label{sec:PathSpaces}

Throughout, we fix a pointed diffeological space $(X, \bullet)$ which is path-connected, i.e. for every $x,y \in X$, there exists an $\alpha \in \mathcal P X$ for which $x,y \in \image(\alpha)$. Note that this fact implies there exists a smooth map $p : [0,1] \to X$ with $p(0) = x$ and $p(1) = y$. The goal of this section is to provide a complete and rigorous framework for Singer's arguments in \citep[\S3]{Singer} and to generalize the notions to the diffeological setting.

In order to achieve our goal, we must carefully develop the spaces $\mathcal F_\bullet(X)$ and $\Omega_\bullet(X)$, as well as a path-concatenation operation. Given two maps $\alpha, \beta : [0,1] \to X$ with $\beta(0) = \alpha(1)$, we define their concatenation $\beta \vee \alpha : [0,1] \to X$ in the usual way:
\[
    (\beta \vee \alpha)(t) =
    \begin{dcases}
        \alpha(2t) & \text{if } 0 \leq t \leq \frac{1}{2}. \\
        \beta(2t - 1) & \text{if } \frac{1}{2} \leq t \leq 1.
    \end{dcases}
\]
This is essentially Definition \ref{def:OpenPathConcatenation} for maps defined on the closed interval $[0,1]$. However, one technical issue is that even if $\alpha$ and $\beta$ are smooth, their concatenation in general will not be. To be able to concatenate smooth paths into other smooth paths is essentially required, since we want to formulate Singer's bundle $\tau : \mathcal F_\bullet(X) \to X$ as a diffeological principal $\Omega_\bullet(X)$-bundle, in which case $\Omega_\bullet(X)$ needs to be a diffeological group under $\vee$ and, furthermore, the $\Omega_\bullet(X)$-action needs to be defined.

We will use the \textbf{retrace-equivalence} approach, as it seems to be the most straightforward to implement in the diffeological setting. Diffeology handles quotients well, which makes the retrace-equivalence approach particularly natural. However, in order to make this precise, we will need some tools from analysis which we discuss in the following Section \ref{subsec:SmoothRampFunctions}. The uninterested reader may simply look at the definitions and results in Section \ref{subsec:SmoothRampFunctions} and skip the proofs, which are technical but not conceptually difficult. In Section \ref{subsec:PathConcatenation}, we will define the diffeological space $\mathcal F(X)$ (resp. $\Omega_\bullet(X)$) of retrace-equivalent paths (resp. retrace-equivalent loops) and show that path-concatenation is well-defined and smooth on $\mathcal F(X)$. We apply these constructions in Section \ref{subsec:UniversalConnection} to formalize Singer's universal connection in the diffeological setting.

\subsection{Smooth Ramp Functions}
\label{subsec:SmoothRampFunctions}

Everything here derives from the celebrated smooth transition function $\uparrow\nobreak\ : \R \to \R$ defined by
\[
    \uparrow\!(t) =
    \begin{dcases}
        0 & \text{if } t \leq 0, \\
        \frac{\exp(-\frac{1}{t})}{\exp(-\frac{1}{t}) + \exp(- \frac{1}{1-t})} & \text{if } 0 < t < 1, \\
        1 & \text{if } t \geq 1.
    \end{dcases}
\]
In general, for real numbers $a < b$ and $\varepsilon > 0$, we have the following \emph{ramp-up} functions
\[
\uparrow_a^b\!(t, \varepsilon) = a + (b - a)\uparrow\!\left(\frac{t}{\varepsilon} + \frac{1}{2} \right)
\]
and the following \emph{ramp-down} functions
\[
\downarrow_a^b\!(t, \varepsilon) = a + b - \uparrow_a^b\!(t, \varepsilon).
\]
For brevity, we write $\uparrow\nobreak\!(t, \varepsilon)$ for $\uparrow_0^1\nobreak\!(t, \varepsilon)$ and $\downarrow\nobreak\!(t, \varepsilon)$ for $\downarrow_0^1\nobreak\!(t, \varepsilon)$.
\begin{proposition}
    The smooth ramp functions satisfy the following properties:
    \begin{enumerate}
        \item $\uparrow_a^b\!$ and $\downarrow_a^b\!$ are $\smooth$.
        \item $\uparrow_a^b\!$ is weakly increasing and $\downarrow_a^b\!$ is weakly decreasing.
        \item $\uparrow_a^b\nobreak\ \equiv a$ on $(-\infty, -\varepsilon]$ and $\uparrow_a^b\nobreak\ \equiv b$ on $[\varepsilon, \infty)$.
        \item $\downarrow_a^b\nobreak\ \equiv b$ on $(-\infty, -\varepsilon]$ and $\downarrow_a^b\nobreak\ \equiv a$ on $[\varepsilon, \infty)$.
    \end{enumerate}
\end{proposition}

It will be important in Section \ref{subsec:PathConcatenation} to ``level out'' smooth, weakly increasing functions $g : [a, b] \to \R$ into smooth, weakly increasing functions $g^R : [a, b] \to \R$ which are constant near the right-endpoint $b$ and otherwise agree with $g$. The superscript $R$ represents this behavior. Similarly, we can smoothly level out on the left side to obtain $g^L$. More precisely:
\begin{lemma}
    \label{lem:LevelOut}
    Let $g : [a,b] \to \R$ be a smooth, weakly increasing function. Then, for any $t_0 \in (a,b)$ and small $\varepsilon > 0$, there exists functions $g^R_\varepsilon, g^L_\varepsilon : [a,b] \to \R$ such that:
    \begin{enumerate}
        \item Both $g^R_\varepsilon$ and $g^L_\varepsilon$ are $\smooth$.
        \item Both $g^R_\varepsilon$ and $g^L_\varepsilon$ are weakly increasing.
        \item $g^R_\varepsilon \equiv g$ on $[a, t_0 - \varepsilon]$ and $g^R_\varepsilon$ is constant on $[t_0 + \varepsilon, b]$. 
        \item $g^L_\varepsilon$ is constant on $[a, t_0 - \varepsilon]$ and $g^L_\varepsilon \equiv g$ on $[t_0 + \varepsilon, b]$.
    \end{enumerate}
\end{lemma}
\begin{proof}
    Let $t_0$ and $\varepsilon > 0$ be as above. The desired functions are $g^R_\varepsilon : [a, b] \to \R$ given by
    \[
    g^R_\varepsilon(t) = g(a) + \int_{a}^t g'(s) \downarrow\!(s - t_0, \varepsilon)\ ds
    \]
    and $g^L_\varepsilon : [a,b] \to \R$ given by 
    \[
    g^L_\varepsilon(t) = g(b) - \int_{t}^b g'(s) \uparrow\!(s - t_0, \varepsilon)\ ds.
    \]

\end{proof}

\begin{corollary}
    \label{cor:SmoothingFunctions}
    Let $f : [a,c] \to \R $ be a weakly increasing, real-valued function such that there exists $b \in (a,c)$ and smooth functions $f_1 : [a,b] \to \R$ and $f_2 : [b,c]\to \R$ for which
    \[
    f(t) =
    \begin{dcases}
        f_1(t) & \text{if } t \in [a,b], \\
        f_2(t) & \text{if } t \in [b, c].
    \end{dcases}
    \]
    Then, for any small $\varepsilon > 0$, there is a smooth, weakly increasing function $f_\varepsilon^\infty$ defined on $[a,c]$ such that $f_\varepsilon^\infty \equiv f$ on $[a,c] - (b - \varepsilon, b + \varepsilon)$. If $f$ is surjective, so also is $f_\varepsilon^\infty$.
\end{corollary}
\begin{proof}
    By hypothesis, there are smooth, weakly increasing functions $f_1 : [a,b] \to \R$ and $f_2 : [b,c] \to \R$ such that 
    \[
    f(t) =
    \begin{dcases}
        f_1(t) & \text{if } t \in [a,b]. \\
        f_2(t) & \text{if } t \in [b, c].
    \end{dcases}
    \]
    Let $\varepsilon > 0$ be small. Choose $t_0 = b - \varepsilon/2$. Applying Lemma \ref{lem:LevelOut}, we obtain a smooth, weakly increasing function $(f_1)^R_{\varepsilon/4} : [a,b] \to \R$ which agrees with $f_1$ on $[a, t_0 - \varepsilon/4] = [a, b - 3\varepsilon/4]$ and is constant on $[t_0 + \varepsilon/4, b] = [b - \varepsilon/4, b]$.

    Similarly, by applying Lemma \ref{lem:LevelOut} at the point $b + \varepsilon/2$, we obtain a smooth, weakly increasing function $(f_2)^L_{\varepsilon/4} : [b,c] \to \R$ which is constant on $[b, b + \varepsilon/4]$ and agrees with $f_2$ on $[b + 3\varepsilon/4, c]$.

    Let $k_1 = (f_1)^R_{\varepsilon/4}(b)$ and $k_2 = (f_2)^L_{\varepsilon/4}(b)$. Define $f^\infty_\varepsilon : [a, c] \to \R$ by the formula
    \[
    f^\infty_\varepsilon(t) =
    \begin{dcases}
        (f_1)^R_{\varepsilon/4}(t) & \text{if } t \in [a, b - \varepsilon/4]. \\
        \uparrow_{k_1}^{k_2}(t - b, \varepsilon/4) & \text{if } t \in [b - \varepsilon/4, b + \varepsilon/4]. \\
        (f_2)^L_{\varepsilon/4}(t) & \text{if } t \in [b + \varepsilon/4, c].
    \end{dcases}
    \] 
    Then $f^\infty_\varepsilon$ is smooth, weakly increasing, and agrees with $f$ on $[a,c] - (b - \varepsilon, b + \varepsilon)$. If $f$ is surjective, then so is $f^\infty_\varepsilon$ by the intermediate value theorem.

\end{proof}

\subsection{Path Concatenation and the Loop Space}
\label{subsec:PathConcatenation}

Equip the set $\smooth([0,1], X)$ with the functional diffeology. We will take two quotients of $\smooth([0,1], X)$ by equivalence relations which will result in a diffeological space of retrace-equivalent classes of paths on which path-concatenation is diffeologically smooth.

The first quotient comes from the standard treatment of paths with ``sitting instants'' (e.g., \citet{CaetanoPicken,CLS2017}). Let $\mathcal M \se \smooth([0,1], [0,1])$ be the monoid of smooth, weakly increasing, surjective maps $f : [0,1] \to [0,1]$. On $\smooth([0,1], X)$, define a relation $\alpha \sim \beta$ if and only if there is a finite sequence $\alpha = \eta_0, \eta_1, \dots, \eta_n = \beta$ such that for each $1 \leq i \leq n$, one has that either
\begin{enumerate}
    \item $\eta_i = \eta_{i - 1} \circ f_i$, or
    \item $\eta_{i - 1} = \eta_i \circ f_i$ for $f_i \in \mathcal M$
\end{enumerate}
for $f_i \in \mathcal M$. Then $\sim$ is an equivalence relation. 

Let $\smooth_{\sim}([0,1], X) = \smooth([0,1], X)/\!\sim$ be the diffeological quotient space. An element of $\smooth_{\sim}([0,1], X)$ is written $\corners{\alpha}$.
\begin{lemma}
    \label{lem:ConstantNearEndPoint}
    Let $\corners{\alpha} \in \smooth_{\sim}([0,1], X)$. Then the representative $\alpha : [0,1] \to X$ can be chosen to be constant near its endpoints $0$ and $1$.
\end{lemma}
\begin{proof}
    Take, for example, the smooth function $f : [0,1] \to [0,1]$ given by $f(t) =\ \uparrow\nobreak\!(t - \frac{1}{2}, \varepsilon = 1/4)$. Then $f \in \mathcal M$, and $\alpha \circ f$ is constant near $0$ and $1$.

\end{proof}

Observe that if $\alpha, \beta : [0,1] \to X$ are smooth paths which are constant near $0$ and $1$ such that $\beta(0) = \alpha(1)$, then their concatenation $\beta \vee \alpha : [0,1] \to X$ is again smooth. Due to Lemma \ref{lem:ConstantNearEndPoint}, we may consider descending the path-concatenation operation $\vee$ into the quotient $\smooth_{\sim}([0,1], X)$ by defining
\[
    \corners{\beta} \vee \corners{\alpha} := \corners{\beta_c \vee \alpha_c},
\]
where $\beta_c$ and $\alpha_c$ are representatives of $\corners{\beta}$ and $\corners{\alpha}$ which are constant near $0$ and $1$. Thanks to the set-up of Section \ref{subsec:SmoothRampFunctions}, this is well-defined:
\begin{proposition}
    \label{prop:ConcatenationWellDefined}
    The path-concatenation operation $\vee$ is well-defined on $\smooth_{\sim}([0,1], X)$. More explicitly, given $\corners{\alpha}, \corners{\beta} \in \smooth_{\sim}([0,1], X)$ with $\beta(0) = \alpha(1)$, one has
    \[
    \corners{\beta_c \vee \alpha_c} = \corners{\beta_c' \vee \alpha_c'\ }
    \] 
    for any representatives $\beta_c, \beta_c'$ of $\corners{\beta}$ and $\alpha_c, \alpha_c'$ of $\corners{\alpha}$ which are constant near $0$ and $1$.
\end{proposition}
\begin{proof}
    It suffices to prove the above claim for the case $\alpha_c' = \alpha_c \circ f$ and $\beta_c' = \beta_c$ for some $f \in \mathcal M$. Define $F : [0,1] \to [0,1]$ by the formula
    \[
    F(t) =
    \begin{dcases}
        \frac{1}{2}f(2t) & \text{if } 0 \leq t \leq \frac{1}{2}. \\
        t & \text{if } \frac{1}{2} \leq t \leq 1.
    \end{dcases}
    \]
    Then $F$ is weakly increasing and surjective which is smooth everywhere except possibly at $t = 1/2$. By Corollary \ref{cor:SmoothingFunctions}, there exists a smooth, weakly increasing, surjective function $F_\varepsilon^\infty : [0,1] \to [0,1]$ which agrees with $F$ on $[0,1] - (\frac{1}{2} - \varepsilon, \frac{1}{2} + \varepsilon)$ for some small $\varepsilon > 0$. In particular, taking $\varepsilon$ small enough so that $\beta_c \vee \alpha_c$ is constant on $(\frac{1}{2} - \varepsilon, \frac{1}{2} + \varepsilon)$, we have that
    \[
        \beta_c \vee \alpha_c' = (\beta_c \vee \alpha_c) \circ F_\varepsilon^\infty,
    \]
    which proves the claim.

\end{proof}

\begin{convention*}
    At this point, we will for the sake of brevity assume that a given class $\corners{\alpha} \in \smooth_\sim([0,1], X)$ is always written so that the representative $\alpha$ is constant near $0$ and $1$, which is justified by Lemma \ref{lem:ConstantNearEndPoint}. In particular, we write $\corners{\beta} \vee \corners{\alpha} = \corners{\beta \vee \alpha}$ instead of $\corners{\beta} \vee \corners{\alpha} = \corners{\beta_c \vee \alpha_c}$.
\end{convention*}

\begin{proposition}
    \label{prop:PathConcatenationAssociative}
    Path-concatenation is associative on $\smooth_\sim([0,1], X)$. That is,
    \[
        \big(\corners{\gamma} \vee \corners{\beta}\big) \vee \corners{\alpha} = \corners{\gamma} \vee \big( \corners{\beta} \vee \corners{\alpha} \big)
    \]
    if $\beta(0) = \alpha(1)$ and $\gamma(0) = \beta(1)$.
\end{proposition}
\begin{proof}
    Define the function $f : [0,1] \to [0,1]$ by
    \[
    f(t) = \begin{dcases}
        2t & \text{if } 0 \leq t \leq \frac{1}{4}. \\
        t + \frac{1}{4} & \text{if } \frac{1}{4} \leq t \leq \frac{1}{2}. \\
        \frac{1}{2}(t + 1) & \text{if } \frac{1}{2} \leq t \leq 1.
    \end{dcases}
    \]
    Then $f$ is weakly increasing, surjective and satisfies
    \[
    \big( (\gamma \vee \beta) \vee \alpha \big) \circ f = \gamma \vee (\beta \vee \alpha).
    \]
    The claim follows by after applying Corollary \ref{cor:SmoothingFunctions} twice: once to the piecewise function $g : [0, \frac{1}{2}] \to [0, \frac{3}{4}]$ defined by
    \[
    g(t) = \begin{dcases}
        2t & \text{if } 0 \leq t \leq \frac{1}{4}. \\
        t + \frac{1}{4} & \text{if } \frac{1}{4} \leq t \leq \frac{1}{2}.
    \end{dcases}
    \]
    at $t = 1/4$ to obtain $g_\varepsilon^\infty$, and then again to the piecewise function $h : [0, 1] \to [0, 1]$ defined by
    \[
    h(t) = \begin{dcases}
        g_\varepsilon^\infty(t) & \text{if } 0 \leq t \leq \frac{1}{2}. \\
        \frac{1}{2}(t + 1) & \text{if } \frac{1}{2} \leq t \leq 1.
    \end{dcases}
    \]
    at $t = 1/2$ to obtain $h_\varepsilon^\infty$, where $\varepsilon > 0$ is taken appropriately small. Then $h_\varepsilon^\infty \in \mathcal M$ and agrees with $f$ everywhere except on $\varepsilon$-neighborhoods of $\frac{1}{4}$ and $\frac{1}{2}$ (where $(\gamma \vee \beta) \vee \alpha$ are constant). In particular, we have that
    \[
        \big( (\gamma \vee \beta) \vee \alpha \big) \circ h_\varepsilon^\infty = \gamma \vee (\beta \vee \alpha),
    \]
    which proves the claim.

\end{proof}

For any map $\eta : [0,1] \to X$, let $\eta^\leftarrow : [0,1] \to X$ be the map defined by $\eta^\leftarrow(t) = \eta(1 - t)$, which is the path $\eta$ traversed in the opposite direction (this is the \textit{reversal} of Definition \ref{def:OpenPathConcatenation}(ii) for paths on $[0,1]$). Note that if $\eta$ is smooth, then so is $\eta^\leftarrow$. Also denote by $c_x : [0,1] \to X$ the constant path at $x \in X$ defined by $c_x(t) = x$ for all $t \in [0,1]$. We are now in a position to begin defining the retrace-equivalence relation.

\begin{definition}
    \label{def:Decomposition}
    Let $\corners{\alpha} \in \smooth_{\sim}([0,1], X)$.
    \begin{enumerate}
        \item A \textbf{decomposition} of $\corners{\alpha}$ is a finite sequence $\corners{\alpha_1}, \dots, \corners{\alpha_n}$ in $\smooth_{\sim}([0,1], X)$ such that
        \[
            \corners{\alpha} = \corners{\alpha_n} \vee \cdots \vee \corners{\alpha_1}.
        \] 
        \item A \textbf{decomposition point} of $\corners{\alpha}$ is a pair $(\corners{\gamma}, \corners{\beta})$ such that there is some decomposition 
        \[
            \corners{\alpha} = \corners{\alpha_n} \vee \cdots \vee \corners{\alpha_1}.
        \] 
        with $\corners{\gamma} = \corners{\alpha_{i + 1}}$ and $\corners{\beta} = \corners{\alpha_i}$ for some $1 \leq i < n$.
        \item Let $\corners{\alpha_n} \vee \cdots \vee \corners{\alpha_1}$ be some decomposition, and let $\corners{\eta} \in \smooth_{\sim}([0,1], X)$ be such that $\eta(0) = \alpha_i(1)$ for some $i$. By an \textbf{insertion of a factor} $\corners{\eta^\leftarrow} \vee \corners{\eta}$ at the decomposition point $(\corners{\alpha_{i + 1}}, \corners{\alpha_i})$, we mean the new decomposition
        \[
            \corners{\alpha_n} \vee \cdots \vee \corners{\alpha_{i + 1}} \vee \corners{\eta^\leftarrow} \vee \corners{\eta} \vee \corners{\alpha_i} \vee \cdots \vee \corners{\alpha_1}.
        \]
        \item Let $\corners{\alpha_n} \vee \cdots \vee \corners{\alpha_1}$ be some decomposition such that at some $i$, we have $\corners{\alpha_{i + 1}} \vee \corners{\alpha_i} = \corners{\gamma^\leftarrow} \vee \corners{\gamma}$. By a \textbf{deletion of a factor} $\corners{\gamma^\leftarrow} \vee \corners{\gamma}$ at the decomposition point $(\corners{\alpha_{i + 1}}, \corners{\alpha_i})$, we mean the new decomposition
        \[
            \corners{\alpha_n} \vee \cdots \vee \corners{\alpha_{i + 2}} \vee \corners{c_{\gamma(0)}} \vee \corners{\alpha_{i-1}} \vee \cdots \vee \corners{\alpha_1}.
        \]
    \end{enumerate}
\end{definition}

Note that insertions or deletions of factors $\corners{\eta^\leftarrow} \vee \corners{\eta}$ at decomposition points do not, in general, give new decompositions of the \emph{same} element of $\smooth_\sim([0,1], X)$. For example, $\corners{\alpha_2} \vee \corners{\eta^\leftarrow} \vee \corners{\eta} \vee \corners{\alpha_1}$ is not necessarily equal to $\corners{\alpha_2} \vee \corners{\alpha_1}$ in $\smooth_\sim([0,1], X)$. The idea of retrace-equivalence is that these should be identified:
\begin{definition}
    \label{def:RetraceEquivalence}
    Define an equivalence relation $\approx$ on $\smooth_{\sim}([0,1], X)$ by $\corners{\alpha} \approx \corners{\beta}$ if and only if $\corners{\alpha}$ is obtained from $\corners{\beta}$ by finitely many (possibly zero) insertions or deletions of factors of the form $\corners{\eta^\leftarrow} \vee \corners{\eta}$ at decomposition points. If $\corners{\alpha} \approx \corners{\beta}$, we say that $\corners{\alpha}$ and $\corners{\beta}$ are \textbf{retrace-equivalent}. The diffeological space of \textbf{retrace-equivalent paths} is denoted by 
    \[
        \mathcal F(X) := \smooth_{\sim}([0,1], X)/\!\approx,
    \]
    equipped with the quotient diffeology.
\end{definition}

Since $\mathcal F(X)$ is a double quotient from $\smooth([0,1], X)$, we may view a member of $\mathcal F(X)$ as having a representative $\alpha$ from $\smooth([0,1], X)$ instead of the more-technically-correct representative $\corners{\alpha}$ from $\smooth_{\sim}([0,1], X)$. Therefore, a class of retrace-equivalent paths will be denoted by symbols like $[\alpha] \in \mathcal F(X)$, where $\alpha \in \smooth([0,1], X)$. We can then speak of two paths $\alpha, \beta \in \smooth([0,1], X)$ being retrace-equivalent and write $\alpha \approx \beta$ instead of $\corners{\alpha} \approx \corners{\beta}$. Furthermore, in light of Lemma \ref{lem:ConstantNearEndPoint}, our earlier convention will descend here to $\mathcal F(X)$:
\begin{convention*}
    Given $[\alpha] \in \mathcal F(X)$, we assume the representative $\alpha \in \smooth([0,1], X)$ is constant near $0$ and $1$.
\end{convention*}

Let $\corners{\alpha}, \corners{\beta} \in \smooth_{\sim}([0,1], X)$ be such that $\beta(0) = \alpha(1)$. Observe that
\[
    \corners{\beta} \vee \corners{\alpha} = \corners{\beta} \vee \corners{\alpha_{n}} \vee \cdots \vee \corners{\alpha_{1}}
\]
is always a decomposition of $\corners{\beta} \vee \corners{\alpha}$ for any decomposition $\corners{\alpha_{n}} \vee \cdots \vee \corners{\alpha_{1}}$ of $\corners{\alpha}$. In particular, this implies that if $\corners{\alpha} \approx \corners{\alpha'\ }$, then $\corners{\beta} \vee \corners{\alpha} \approx \corners{\beta} \vee \corners{\alpha'\ }$ also. Path-concatenation $\vee$ therefore descends further to $\mathcal F(X)$, by setting
\[
    [\beta] \vee [\alpha] := [ \beta \vee \alpha],
\]
which is well-defined by what was just discussed.
\begin{proposition}
    \label{prop:FXGroupoid}
    The pair $(\mathcal F(X), \vee)$ defines a smooth groupoid. That is, for all $[\alpha], [\beta], [\gamma] \in \mathcal F(X)$ with $\beta(0) = \alpha(1)$ and $\gamma(0) = \beta(1)$, one has:
    \begin{enumerate}
        \item Associativity: $([\gamma] \vee [\beta]) \vee [\alpha] = [\gamma] \vee ([\beta] \vee [\alpha])$.
        \item Units:  $[c_{\alpha(1)}] \vee [\alpha] = [\alpha] = [\alpha] \vee [c_{\alpha(0)}]$.
        \item Inverses: $[\alpha^\leftarrow] \vee [\alpha] = [c_{\alpha(0)}]$ and $[\alpha] \vee [\alpha^\leftarrow] = [c_{\alpha(1)}]$.
    \end{enumerate}
    and that path-concatenation defines a partial function
    \[
        \mathcal F(X) \times \mathcal F(X) \rightharpoonup \mathcal F(X), \qquad ([\alpha], [\beta]) \mapsto [\beta] \vee [\alpha],
    \]
    which is smooth where it is defined: on the diffeological subspace
    \[
        \mathcal F^\vee(X) = \{\big([\alpha], [\beta]\big) \in \mathcal F(X) \times \mathcal F(X) \mid \beta(0) = \alpha(1) \}.
    \]
\end{proposition}
\begin{proof}\
    
    \begin{enumerate}
        \item Associativity: by Proposition \ref{prop:PathConcatenationAssociative}, we have that
        \begin{align*}
            \big([\gamma] \vee [\beta]\big) \vee [\alpha] &= \big[(\corners{\gamma} \vee \corners{\beta}) \vee \corners{\alpha} \big] \\
            &= \big[ \corners{\gamma} \vee (\corners{\beta} \vee \corners{\alpha}) \big] \\
            &= [\gamma] \vee \big([\beta] \vee [\alpha]\big).
        \end{align*}
        \item Units: let $x = \alpha(0)$. Then the function $f : [0,1] \to [0,1]$ defined by the formula
        \[
        f(t) = \begin{dcases}
            0 & \text{if } 0 \leq t \leq \frac{1}{2} \\
            2t - 1 & \text{if } \frac{1}{2} \leq t \leq 1
        \end{dcases}
        \]
        is a weakly increasing, surjective map such that $\alpha \circ f = \alpha \vee c_x$ and is smooth everywhere except at $t = 1/2$. Since $\alpha$ is constant near $t = 0$, choose $\varepsilon > 0$ small enough so that $\alpha \vee c_x$ is constant on $[0, 1/2 + \varepsilon]$. An application of Corollary \ref{cor:SmoothingFunctions} at $t = 1/2$ gives a function $f_\varepsilon^\infty \in \mathcal M$ such that $\alpha \circ f_\varepsilon^\infty = \alpha \vee c_x$, hence $\alpha \sim \alpha \vee c_x$. In particular, $\alpha \approx \alpha \vee c_x$ or equivalently $[\alpha] \vee [c_x] = [\alpha$]. The argument for $[c_{\alpha(1)}] \vee [\alpha] = [\alpha]$ is similar.
        \item Inverses: follows directly by definition of $\approx$.
    \end{enumerate}

    We now show that $\vee$ is smooth. Let $r \mapsto ([\alpha_r], [\beta_r])$ be a plot for the subspace $\mathcal F^\vee(X)$. Locally, by definition of the quotient diffeology, there are plots $r \mapsto \alpha_r$ and $r\mapsto \beta_r$ for $\smooth([0, 1], X)$ with $\alpha_r(1) = \beta_r(0)$ obtained by a 2-step local lift. Note that the endpoints of the representatives $\alpha_r$ and $\beta_r$ are well-defined. We can assume \emph{all} the $\alpha_r$ and $\beta_r$ are constant near their endpoints, for otherwise by smoothness of composition (Lemma \ref{lem:CompositionIsSmooth}), we can consider the plots $r \mapsto \alpha_r \circ f$ and $r\mapsto \beta_r \circ f$, where $f$ is the smooth function in the proof of Lemma \ref{lem:ConstantNearEndPoint}. Now, the evaluation
    \[
        (r, t) \mapsto (\beta_r \vee \alpha_r)(t)
    \]
    is smooth, since all three formulas agree on overlaps thanks to the sitting instants (note that the width of a sitting instant can be chosen uniformly in $r$). This means $r \mapsto \beta_r \vee \alpha_r$ is a plot for $\smooth([0,1], X)$, hence $r \mapsto [\beta_r \vee \alpha_r]$ is a plot for $\mathcal F(X)$. We've just shown that $([\alpha], [\beta]) \mapsto [\beta \vee \alpha]$ locally takes plots for $\mathcal F^\vee(X)$ to plots for $\mathcal F(X)$, hence defines a smooth map.
    
\end{proof}

\begin{definition}
    \label{defLoopSpace}
    The \textbf{loop space} of $(X, \bullet)$ is the set
    \[
        \Omega_\bullet(X) := \big\{ [\gamma] \in \mathcal F(X) : \gamma(0) = \bullet = \gamma(1) \big\}
    \]
    endowed with the subspace diffeology from $\mathcal F(X)$.
\end{definition}

\begin{proposition}
    \label{claimLoopSpaceDiffeologicalGroup}
    The loop space $\Omega_\bullet(X)$ is a diffeological group whose group multiplication and inversion operations are $[\gamma] \vee [\sigma] := [\gamma \vee \sigma]$ and $[\gamma]\inverse = [\gamma^\leftarrow]$.
\end{proposition}
\begin{proof}
    The fact that $\Omega_\bullet(X)$ is a group follows from Proposition \ref{prop:FXGroupoid}. Smoothness of the group multiplication also follows from Proposition \ref{prop:FXGroupoid}. We can show that inversion is smooth. Given a plot $r \mapsto [\gamma_r]$ for $\Omega_\bullet(X)$, the composition
    \[
        (r, t) \mapsto (r, 1 - t) \mapsto \gamma_r(1 - t) = \gamma_r^\leftarrow(t)
    \]
    is smooth, since the plot $r \mapsto [\gamma_r]$ for the quotient has a local 2-step lift into the plot $r \mapsto \gamma_r$ for $\smooth([0,1], X)$. This means that $r \mapsto [\gamma_r]\inverse = [\gamma_r^\leftarrow]$ is a plot for $\Omega_\bullet(X)$, finishing the proof.

\end{proof}

\begin{remark}
    \label{rem:OnRetraceLoops}
    The anonymous referee asks what $\Omega_\bullet(X)$ looks like for a non-manifold $X$, say an irrational torus or an orbifold. The author does not yet have a good answer and leaves the question as a remark. Even the case $X = \R$ appears difficult; Lemma \ref{lem:ReparameterizationInvarianceOfRetrace} says essentially that smooth $\R$-valued loops which have \textit{finitely-many} strict extrema are all retrace-equivalent to a constant loop. Understanding the objects of $\Omega_\bullet(X)$ would lead to a better understanding of the categories in Section \ref{subsec:CategoryEquivalence} and the results relating to them.
\end{remark}

Let $\pi : (E, \xi_0) \to (X, x_0)$ be a diffeological principal $G$-bundle equipped with a connection $A : \Pin E \to \mathcal P E$. We finish the section by describing how the holonomy representation $H^A_{\xi_0} : \mathcal L(X, x_0) \to G$ yields a well-defined, smooth group homomorphism $\Omega_{x_0}(X) \to G$.

\begin{remark}
    \label{rem:OnOpenVsClosedPaths}
    The definitions of diffeology, such as the holonomy representation we are about to discuss, most naturally apply to paths defined on \textit{open} intervals. One may appropriately wonder why we are bothering with paths on \textit{closed} intervals then, in particular paths defined on $[0,1]$. The answer is just that $\smooth([0,1], X)$ is much simpler to work with when viewed as a \textit{diffeological space} to be studied rather than as an auxiliary object of some diffeological definition (e.g., a diffeological connection, which involves horizontally projecting elements of $\mathcal P(\cdot)$, which are paths on arbitrary open intervals). For example, a smooth path (in the sense of diffeology) into $\smooth([0,1], X)$ is simply a smooth family of paths all defined on the common interval $[0,1]$ (compare to Definition \ref{def:FunctionalDiffeologyOfDiffeology}), but the smooth family itself is parameterized by some arbitrary open interval that is usually irrelevant. In short, we view $\smooth([0,1], X)$ as just a diffeological space $Y$ possessing certain properties, whereas $\mathcal P Y$ is used exclusively for diffeological notions, like connections. See Section \ref{subsec:UniversalConnection}.
\end{remark}

Let $\smooth_c([0,1], X)$ be the diffeological subspace of $\smooth([0,1], X)$ of paths with ``sitting instants,'' i.e. each $\alpha \in \smooth_c([0,1], X)$ is constant near $0$ and $1$. Given $\alpha \in \smooth_c([0,1], X)$, define $\hat \alpha : \R \to X$ as the unique smooth path for which $\hat \alpha \equiv \alpha$ on $[0,1]$ and is constant on $(-\infty, 0] \cup [1, \infty)$.

\begin{lemma}
    \label{lem:ExtensionMapIsSmooth}
    The map $\smooth_c([0,1], X) \to \P X$ given by $\alpha \mapsto \hat \alpha$ is smooth.
\end{lemma}
\begin{proof}
    Let $U \ni r \mapsto \alpha_r$ be a plot for $\smooth_c([0,1], X)$. This means that the map $U \times [0,1] \to X$ given by $(r,t) \mapsto \alpha_r(t)$ is smooth. Now, the map $U \times \R \to X$ given by $(r,t) \mapsto \hat \alpha_r(t)$ is locally either constant or equal to $\alpha_r(t)$, so is itself smooth. Therefore, $r \mapsto \hat \alpha_r$ is a plot, meaning that $\alpha \mapsto \hat \alpha$ is smooth.

\end{proof}

\begin{lemma}
    \label{lem:HolonomyPreservesRetrace}
    Let $\pi : (E, \xi_0) \to (X, x_0)$ be a diffeological principal $G$-bundle equipped with a connection $A : \Pin E \to \P E$, and let $H^A_{\xi_0} : \mathcal L(X, x_0) \to G$ be its associated holonomy representation. Then:
    \begin{enumerate}[(1)]
        \item $H^A_{\xi_0}(\widehat{\gamma \circ f}) = H^A_{\xi_0}(\hat\gamma)$ for all $\gamma \in \mathcal C^\infty_c([0,1], X)$ a loop and $f \in \mathcal M$.
        \item $H^A_{\xi_0}(\widehat{\eta^\leftarrow \vee \eta}) = e$ for all $\eta \in \mathcal C^\infty_c([0,1], X)$.
    \end{enumerate}
\end{lemma}
\begin{proof}\

    \begin{enumerate}[(1)]
        \item Fix $f \in \mathcal M$. By an application of Lemma \ref{lem:LevelOut}, there exists $f_\varepsilon \in \mathcal M$ such that $f_\varepsilon \equiv 0$ on $[0, \varepsilon/2]$, $f_\varepsilon \equiv 1$ on $[1 - \varepsilon/2, 1]$, and $f_\varepsilon \equiv f$ on $[\varepsilon, 1 - \varepsilon]$, where $\varepsilon > 0$ is chosen so that $\gamma$ is constant on $[0, \varepsilon] \cup [1 - \varepsilon, 1]$. Let $\hat f_\varepsilon : \R \to \R$ denote the unique smooth extension of $f_\varepsilon$ to all of $\R$ by setting $\hat f_\varepsilon \equiv 0$ for $t \leq 0$ and $\hat f_\varepsilon \equiv 1$ for $t \geq 1$. Then,
        \[
            \widehat{(\gamma \circ f)} = \hat \gamma \circ \hat f_\varepsilon.
        \]
        By the reparameterization condition of $\hor_A$, one has
        \[
            \hor_A(\widehat{\gamma \circ f}, 0, \xi_0) = \hor_A(\hat \gamma \circ \hat f_\varepsilon, 0, \xi_0) = \hor_A(\hat \gamma, 0, \xi_0) \circ \hat f_\varepsilon
        \]
        and evaluating the above at time $t = 1$ gives $\xi_0 \cdot H^A_{\xi_0}(\widehat{\gamma \circ f}) = \xi_0 \cdot H^A_{\xi_0}(\hat \gamma)$, since $\hat f_\varepsilon(1) = 1$. Thus $H^A_{\xi_0}(\widehat{\gamma \circ f}) = H^A_{\xi_0}(\hat\gamma)$.

        \item Observe that $\widehat{\eta^\leftarrow \vee \eta} = \widehat{\eta^\leftarrow} \vee \hat \eta$ and $\widehat{\eta^\leftarrow} = \hat \eta^\leftarrow$. Then, by the concatenation and reversal properties (Proposition \ref{prop:PropertiesOfHorizontalLift}(iii) and (iv)) of $\hor_A$,
        \begin{align*}
            \hor_A(\widehat{\eta^\leftarrow \vee \eta}, 0, \xi_0) &= \hor_A( \hat \eta^\leftarrow, 0, \hor_A(\hat \eta, 0, \xi_0)(1) ) \vee \hor_A(\hat \eta, 0, \xi_0) \\ 
            &= \hor_A(\hat \eta, 0, \xi_0)^\leftarrow \vee \hor_A(\hat \eta, 0, \xi_0)
        \end{align*}
        and evaluating at $t = 1$ gives
        \[
        \xi_0\cdot H^A_{\xi_0}(\widehat{\eta^\leftarrow \vee \eta}) = \hor_A(\hat \eta, 0, \xi_0)^\leftarrow(1) = \hor_A(\hat \eta, 0, \xi_0)(0) = \xi_0.
        \]
        Therefore, $H^A_{\xi_0}(\widehat{\eta^\leftarrow \vee \eta}) = e$.
        
    \end{enumerate}
\end{proof}

\begin{proposition}
    \label{prop:HolonomyOnLoopGroupIsSmooth}
    Let $\pi : (E, \xi_0) \to (X, x_0)$ be a diffeological principal $G$-bundle equipped with a connection $A : \Pin E \to \P E$, and let $H^A_{\xi_0} : \mathcal L(X, x_0) \to G$ be its associated holonomy representation. Then the map
    \[
        \Omega_{x_0}(X) \to G,\quad [\gamma] \mapsto H^A_{\xi_0}(\hat \gamma)
    \]
    for any choice of representative $\gamma \in \smooth_c([0,1], X)$ is a well-defined, smooth homomorphism of diffeological groups. We denote this map using the same symbol $H^A_{\xi_0}$, i.e. $H^A_{\xi_0} : \Omega_{x_0}(X) \to G$ is given by $H^A_{\xi_0}([\gamma]) = H^A_{\xi_0}(\hat \gamma)$.
\end{proposition}
\begin{proof}
    Two representatives $\gamma_1, \gamma_2 \in \smooth_c([0,1], X)$ for the retrace-equivalence class $[\gamma] \in \Omega_{x_0}(X)$ differ by reparameterizations via members of $\mathcal M$ and/or by retrace factors $\eta^\leftarrow \vee \eta$. Lemma \ref{lem:HolonomyPreservesRetrace} together with Lemma \ref{lem:AlmostGroupHomomorphism} then implies that $H^A_{\xi_0}(\hat \gamma_1) = H^A_{\xi_0}(\hat \gamma_2)$, so that $H^A_{\xi_0}[\gamma]$ is well-defined.

    We show now that $H^A_{\xi_0} : \Omega_{x_0}(X) \to G$ is smooth. Let $r \mapsto [\gamma_r]$ be a plot for $\Omega_{x_0}(X)$. Then, there is a local lift to a plot $r \mapsto \corners{\gamma_r}$ for $\smooth_\sim([0,1], X)$ with $\gamma_r(0) = x_0 = \gamma_r(1)$ for all representatives $\gamma_r$ of $\corners{\gamma_r}$. This plot in turn has a local lift $r \mapsto \gamma_r$ for $\smooth([0,1], X)$ with $\gamma_r(0) = x_0 = \gamma_r(1)$. We may assume all the $\gamma_r$ are constant near $0$ and $1$, i.e. that $\gamma_r \in \smooth_c([0,1], X)$, for otherwise we can pre-compose with the function $f$ from the proof of Lemma \ref{lem:ConstantNearEndPoint}, thanks to smoothness of composition (Lemma \ref{lem:CompositionIsSmooth}). By Lemma \ref{lem:ExtensionMapIsSmooth}, the composition
    \[
        r \mapsto [\gamma_r] \mapsto \corners{\gamma_r} \mapsto \gamma_r \mapsto \hat \gamma_r \mapsto H^A_{\xi_0}(\hat \gamma_r) = H^A_{\xi_0}([\gamma_r])
    \]
    is a plot for $G$. Therefore, $H^A_{\xi_0}$ is smooth.

    The homomorphism property of $H^A_{\xi_0} : \Omega_{x_0}(X) \to G$ follows from Lemma \ref{lem:AlmostGroupHomomorphism}:
    \begin{align*}
        H^A_{\xi_0}([\gamma] \vee [\sigma]) &= H^A_{\xi_0}([\gamma \vee \sigma]) \\
        &= H^A_{\xi_0}(\widehat{\gamma \vee \sigma}) \\
        &= H^A_{\xi_0}(\hat \gamma \vee \hat \sigma) \\
        &= H^A_{\xi_0}(\hat \gamma) H^A_{\xi_0}(\hat \sigma) \\
        &= H^A_{\xi_0}([\gamma]) H^A_{\xi_0}([\sigma]).
    \end{align*}

\end{proof}

\subsection{The Universal Connection}
\label{subsec:UniversalConnection}

Adapting Singer's construction in \citep[\S3]{Singer}, we show that every pointed diffeological space $(X, \bullet)$ gives rise to a diffeological principal $\Omega_\bullet(X)$-bundle over itself equipped with a diffeological connection called the \emph{universal connection}.

Let $\mathcal F_\bullet(X) = \{ [\alpha] \in \mathcal F(X) : \alpha(0) = \bullet \}$ be endowed with the subspace diffeology from $\mathcal F(X)$. There is a natural projection 
\[
    \tau : \mathcal F_\bullet(X) \to X, \quad \tau([\alpha]) = \alpha(1),
\]
which is a well-defined smooth surjection, which we may call the \textit{target map}. Then, the map
\[
    \mathcal F_\bullet(X) \times \Omega_\bullet(X) \to \mathcal F_\bullet(X), \quad ([\alpha], [\gamma]) \mapsto [\alpha] \vee [\gamma] = [\alpha \vee \gamma],
\]
defines a smooth right-action of $\Omega_\bullet(X)$ on $\mathcal F_\bullet(X)$, thanks to Proposition \ref{prop:FXGroupoid}. To show $\tau : \mathcal F_\bullet(X) \to X$ is actually a diffeological principal $\Omega_\bullet(X)$-bundle, we will need the following result, which roughly says that plots for $X$ determine locally a smooth family of radial paths in $X$:
\begin{lemma}
    \label{lem:RadialPaths}
    Let $\psi : U \to X$ be a plot and put $x_r := \psi(r)$. Fix a point $r_0 \in U$. Then there exists a neighborhood $V \se U$ of $r_0$ and a plot $V \ni r \mapsto \alpha_r \in \smooth([0,1], X)$ such that $\alpha_r(0) = x_{r_0}$ and $\alpha_r(1) = x_r$ for all $r \in V$.
\end{lemma}
\begin{proof}
    We can choose $V \se U$ to be any open ball containing $r_0$. For each $r \in V$, consider the straight-line path $\ell_r : [0,1] \to V$ from $r_0$ to $r$ in $V$, i.e. $\ell_r$ is given by
    \[
        \ell_r(t) = (1 - t)r_0 + t r.
    \]
    Then $r \mapsto \ell_r$ is a plot for $\smooth([0,1], V)$, since the evaluation $(r,t) \mapsto \ell_r(t)$ is smooth. Define $\alpha_r : [0,1] \to X$ by
    \[
        \alpha_r(t) = \psi(\ell_r(t)).
    \]
    The map $V \ni r \mapsto \alpha_r$ defines the desired plot for $\smooth([0,1], X)$, since the evaluation $(r,t) \mapsto \alpha_r(t) = \psi((1-t)r_0 + tr)$ is a composition of smooth maps.

\end{proof}

\begin{proposition}
    \label{prop:FPrincipalBundle}
    The projection $\tau : \mathcal F_\bullet(X) \to X$ is a diffeological principal $\Omega_\bullet(X)$-bundle.
\end{proposition}
\begin{proof}
    This amounts to showing two things (see Definition \ref{def:DiffeologicalPrincipalBundle}):
    \begin{enumerate}[(i)]
        \item The graph map $\Gamma : \mathcal F_\bullet(X) \times \Omega_\bullet(X) \to \mathcal F_\bullet(X) \times \mathcal F_\bullet(X)$ given by
        \[
            \Gamma([\alpha], [\gamma]) = ([\alpha], [\alpha \vee \gamma])
        \]
        is a diffeomorphism onto its image.
        \item There is a diffeomorphism $\Phi : \mathcal F_\bullet(X) / \Omega_\bullet(X) \to X$ such that 
        \[
            \Phi([\alpha] \vee \Omega_\bullet(X)) = \tau([\alpha]) = \alpha(1),
        \]
        where $[\alpha]\vee \Omega_\bullet(X)$ is the equivalence class 
        \[
            [\alpha] \vee \Omega_\bullet(X) = \big\{ [\alpha \vee \gamma] : [\gamma] \in \Omega_\bullet(X) \big\}
        \]
        in the quotient $\mathcal F_\bullet(X) / \Omega_\bullet(X)$.
    \end{enumerate}

    For (i), the inverse of $\Gamma$, defined on $\image \Gamma$, is given by 
    \[
        \Gamma\inverse([\alpha], [\beta]) = ([\alpha], [\alpha^{\leftarrow} \vee \beta]), \quad \text{where } \alpha^{\leftarrow}(t) = \alpha(1 - t).
    \]
    We can show that the map $[\alpha] \mapsto [\alpha^\leftarrow]$ is smooth. Given a plot $r \mapsto [\alpha_r]$ of $\mathcal F_\bullet(X)$, there is a local lift into a plot $r \mapsto \corners{\alpha_r}$ of $\smooth_\sim([0,1], X)$ with $\alpha_r(0) = \bullet$ for any (all) representatives $\alpha_r$. This lifted plot, in turn, has a local lift into a plot $r \mapsto \alpha_r$ of $\smooth([0,1], X)$ with $\alpha_r(0) = \bullet$. Since $\alpha \mapsto \alpha^\leftarrow$ is a smooth map $\smooth([0,1], X) \to \smooth([0,1], X)$, the composition $r \mapsto \alpha_r \mapsto \alpha_r^\leftarrow$ is a plot for $\smooth([0,1], X)$. Thus, $r \mapsto [\alpha_r^\leftarrow]$ is locally a plot for the subspace of $\mathcal F(X)$ of paths \textit{ending} at $\bullet$, hence $[\alpha] \mapsto [\alpha^\leftarrow]$ is smooth. This together with smoothness of path-concatenation (Proposition \ref{prop:FXGroupoid}) says that $\Gamma\inverse$ is smooth, so $\Gamma$ is a diffeomorphism onto its image.

    For (ii), we need to show that the map $\Phi : \mathcal F_\bullet(X) / \Omega_\bullet(X) \to X$ given by
    \[
        \Phi([\alpha] \vee \Omega_\bullet(X)) = \alpha(1)
    \]
    is smooth with smooth inverse. Any plot for $\mathcal F_\bullet(X) / \Omega_\bullet(X)$ locally has the form $r \mapsto [\alpha_r] \vee \Omega_\bullet(X)$ for a plot $r \mapsto \alpha_r \in \smooth([0,1], X)$, where we are using a 3-step lift due to the triple-quotient nature. Then $r \mapsto \Phi([\alpha_r] \vee \Omega_\bullet(X)) = \alpha_r(1)$ is just the evaluation map $(r,t) \mapsto \alpha_r(t)$ restricted to $t = 1$, which is smooth since $r \mapsto \alpha_r$ is a plot for $\smooth([0, 1], X)$. This means that $\Phi$ is smooth.

    For each $x \in X$, let $\lambda_x : [0,1] \to X$ be some path with $\lambda_x(0) = \bullet$ and $\lambda_x(1) = x$, using the fact that $X$ is path-connected. The inverse $\Phi\inverse : X \to \mathcal F_\bullet(X) / \Omega_\bullet(X)$ is then
    \[
        \Phi\inverse(x) = [\lambda_x] \vee \Omega_\bullet(X).
    \]
    If $[\lambda_x], [\lambda'_x] \in \mathcal F_\bullet(X)$ satisfies $\lambda_x(1) = x = \lambda'_x(1)$, then $[\lambda_x] = [\lambda'_x] \vee [(\lambda'_x)^{\leftarrow} \vee \lambda_x]$, hence $[\lambda_x] \vee \Omega_\bullet(X) = [\lambda'_x] \vee \Omega_\bullet(X)$ and $\Phi\inverse$ is well-defined. It remains to show $\Phi\inverse$ is smooth. Suppose a plot $r \mapsto x_r$ is given. For a point $r_0$ in its domain, there is (by Lemma \ref{lem:RadialPaths}) a plot $r \mapsto \alpha_r$ for $\smooth([0, 1], X)$, defined on an open neighborhood of $r_0$, such that $\alpha_r(0) = x_{r_0}$ and $\alpha_r(1) = x_r$. Then
    \[
        r \mapsto \Phi\inverse(x_r) = [\alpha_r \vee \lambda_{x_{r_0}}] \vee \Omega_\bullet(X)
    \]
    locally defines a plot for $\Omega_\bullet(X)$ (recall the convention that the representatives $\alpha_r$ and $\lambda_{x_{r_0}}$ are chosen to be constant near $0$ and $1$). Since $r_0$ was arbitrary, we've shown that $\Phi\inverse$ is smooth, hence a diffeomorphism.

\end{proof}

\begin{definition}
    \label{def:PointedPathBundle}
    For a path-connected, pointed diffeological space $(X, \bullet)$, the \textbf{pointed path bundle} associated to $(X, \bullet)$ is the principal $\Omega_\bullet(X)$-bundle $\tau : \mathcal F_\bullet(X) \to X$.
\end{definition}

There is a natural notion of parallel transport on the pointed path bundle, which Singer calls the \emph{universal connection} \citep{Singer}. We describe how to adapt Singer's universal connection to the diffeological setting now.

We use the following notation: for any map $\alpha : (a,b) \to X$ and points $r, s \in (a,b)$, we denote by $\alpha^{r \to s}$ to be the map $\alpha^{r \to s} : [0,1] \to X$ given by
\[
    \alpha^{r \to s}(t) = \alpha(ts + (1-t) r).
\]
In other words, $\alpha^{r \to s}$ is the part of $\alpha$ starting at $\alpha(r)$ and ending at $\alpha(s)$, parameterized by the closed interval $[0,1]$. 

\begin{lemma}
    \label{lem:PathSegmentMapIsSmooth}
    The map $\pr_1^*\Pin X \to \smooth([0,1], X)$ given by
    \[
        (\alpha, r, s) \mapsto \alpha^{r\to s}
    \]
    is smooth. Here $\pr_1^*\Pin X = \{ ((\alpha, t), (\beta, s)) \in \Pin X \times \Pin X \mid \alpha = \beta\}$, and we identify an element $((\alpha, r), (\alpha, s))$ with the triple $(\alpha, r, s)$.
\end{lemma}
\begin{proof}
    Let $u \mapsto (\alpha_u, r_u, s_u)$ be a plot for $\pr_1^*\Pin X$. Then evaluation $(\alpha, t) \mapsto \alpha(t)$ is smooth, hence
    \[
        (u, t) \mapsto (u, ts_u + (1 - t) r_u) \mapsto (\alpha_u, ts_u + (1 - t) r_u) \mapsto \alpha_u(ts_u + (1 - t) r_u) = \alpha_u^{r_u \to s_u}(t)
    \]
    is a composition of smooth maps. We've shown that the evaluation $(u, t) \mapsto \alpha_u^{r_u \to s_u}(t)$ is smooth, which means $u \mapsto \alpha_u^{r_u \to s_u}$ is a plot for $\smooth([0,1], X)$.

\end{proof}

\begin{lemma}
    \label{lem:ReparameterizationInvarianceOfRetrace}
    Let $Y \se \R$ be an open interval and $\beta : [0,1] \to Y$ smooth and piecewise weakly monotone. Then $[\beta] = [\ell_{\beta(0),\beta(1)}]$ in $\mathcal F(Y)$, where recall that $\ell_{r,s}$ is the straight-line path from $r$ to $s$. Consequently, $[\alpha \circ \beta] = [\alpha \circ \ell_{\beta(0),\beta(1)}]$ for every smooth $\alpha : Y\to X$.
\end{lemma}
\begin{proof}
    We use induction on the minimal number $m$ of monotone pieces $0 < t_1 < \cdots < t_m < 1$. Suppose $m = 1$. By possibly replacing $\beta$ by $-\beta$, we may assume $\beta$ is weakly increasing on $[0, t_1]$ and weakly decreasing on $[t_1, 1]$. Let $p = \beta(0)$ and $q = \beta(1)$. Also, let $M = \beta(t_1)$; note that $p,q \leq M$, and that either $p \leq q$ or $q \leq p$.
    
    We argue in the case of $p \leq q$ (the other case is similar). By the Intermediate Value Theorem, there exists $t_* \in [0, t_1]$ such that $\beta(t_*) = q$. Let $\eta : [0,1] \to Y$ be the map $\eta = \beta^{t_* \to t_1}$. Note that, by possibly reparameterizing, one has $\corners{\eta^\leftarrow} = \corners{\beta^{t_1 \to 1}}$. Then,
    \[
    \corners{\beta} = \corners{\beta^{t_1 \to 1}} \vee \corners{\beta^{t_* \to t_1}} \vee \corners{\beta^{0 \to t_*}} = \corners{\eta^\leftarrow} \vee \corners{\eta} \vee \corners{\beta^{0 \to t_*}} \approx \corners{\beta^{0 \to t_*}}.
    \]
    But $\beta^{0 \to t_1} : [0,1] \to Y$ is weakly monotone, hence $\corners{\beta^{0 \to t_*}} = \corners{\ell_{p, q}}$. In other words $\beta$ is retrace-equivalent to $\ell_{\beta(0), \beta(1)}$.

    The inductive step is essentially the same argument: if $m \geq 2$ is the minimal number of monotone pieces of $\beta$, then one can show using the above technique that $\beta$ is retrace-equivalent to a path of at most $m - 1$ monotone pieces, hence the claim assuming the inductive hypothesis.
    
\end{proof}

\begin{remark}
    \label{rem:NotEverythingIsRetraceEqualToMonotone}
    The proof of Lemma \ref{lem:ReparameterizationInvarianceOfRetrace} does not apply in general for \textit{arbitrary} $\beta \in \P Y$, for example $\beta(t) = e^{-1/t^2} \sin(1/t)$ extended to $0$ at $t = 0$. This is the main reason for weakening the reparameterization condition of Definition \ref{def:Connection}. See Theorem \ref{thm:UniversalConnectionIsDiffeologicalConnection}
\end{remark}

We are now in a position to describe Singer's universal connection from \citep{Singer}.

\begin{definition}
    \label{def:UniversalConnection}
    The map $\Theta : \Pin \mathcal F_\bullet(X) \to \P \mathcal F_\bullet(X), (\tilde \alpha, r) \mapsto \Theta_r \tilde \alpha$ defined by
    \[
        \Theta_r \tilde \alpha(s) = [(\tau \circ \tilde \alpha)^{r \to s}] \vee \tilde \alpha(r)
    \]
    is called the \textbf{universal connection of $(X, \bullet)$}.
\end{definition}

\begin{theorem}
    \label{thm:UniversalConnectionIsDiffeologicalConnection}
    The universal connection 
    \[
        \Theta : \Pin \mathcal F_\bullet(X) \to \P \mathcal F_\bullet(X), \quad \Theta_r\tilde \alpha(s) = [(\tau \circ \tilde \alpha)^{r \to s}] \vee \tilde \alpha(r)
    \]
    is a diffeological connection on the pointed path bundle $\tau : \mathcal F_\bullet(X) \to X$.
\end{theorem}
\begin{proof}
    We show first that $\Theta : \Pin \mathcal F_\bullet(X) \to \P \mathcal F_\bullet(X)$ is smooth. Let $u \mapsto (\tilde \alpha_u, r_u)$ be a plot for $\Pin\mathcal F_\bullet(X)$; in particular, $u \mapsto \tilde \alpha_u$ is a plot for $\P \mathcal F_\bullet(X)$ and $u \mapsto r_u$ is a plot for $\R$ such that $r_u \in \dom(\tilde \alpha_u)$. Then:
    \begin{enumerate}[(1)]
        \item \textit{The map $u \mapsto \tilde \alpha_u(r_u)$ is a plot (equivalently, smooth).} 
        
        Since $u \mapsto \tilde \alpha_u$ is a plot for $\P \mathcal F_\bullet(X)$, the evaluation map $(u,s) \mapsto \tilde \alpha_u(s)$ is smooth. Then, since $u\mapsto r_u$ is a plot, it follows that the composition $u \mapsto (u, r_u) \mapsto \tilde \alpha_u(r_u)$ is a plot for $\mathcal F_\bullet(X)$.

        \item \textit{The map $(u,s) \mapsto [(\tau \circ \tilde \alpha_u)^{r_u \to s}]$ is smooth.}
        
        Since $\tau : \mathcal F_\bullet(X) \to X$ is a smooth map, the composition $u \mapsto \tau \circ \tilde \alpha_u$ is a plot for $\P X$. By Lemma \ref{lem:PathSegmentMapIsSmooth}, the composition
        \[
            (u,s) \mapsto (\tilde \alpha_u, r_u, s) \mapsto (\tau \circ \tilde \alpha_u, r_u, s) \mapsto (\tau \circ \tilde \alpha_u)^{r_u \to s}
        \]
        is a smooth map. A final composition with the two quotient maps gives that $(u,s) \mapsto [(\tau \circ \tilde \alpha_u)^{r_u \to s}]$ is smooth.
    \end{enumerate}
    Smoothness of path-concatenation (Proposition \ref{prop:FXGroupoid}) together with items (1) and (2) above gives that $(u,s) \mapsto \Theta_{r_u}\tilde \alpha_u(s)$ is smooth, which means that $u \mapsto \Theta_{r_u}\tilde \alpha_u$ is a plot for $\P \mathcal F_\bullet(X)$. Therefore, $\Theta : \Pin \mathcal F_\bullet(X) \to \P \mathcal F_\bullet(X)$ is a smooth map. We now verify the conditions of being a connection (Definition \ref{def:Connection}):
    \begin{enumerate}[(i)]
        \item \emph{Domain}: By definition of $\Theta$.
        \item \emph{Lifting}:
        \[
            \tau(\Theta_r\tilde \alpha(s)) = \tau\big([(\tau \circ \tilde \alpha)^{r \to s}] \vee \tilde \alpha(r) \big) = (\tau \circ \tilde \alpha)^{r \to s}(1) = \tau(\tilde \alpha(s)).
        \]
        \item \emph{Basepoint}: Since $(\tau \circ \tilde \alpha)^{r \to r}$ is the constant path at $\tilde \alpha(r)(1)$,
        \[
            \Theta_r\tilde \alpha(r) = [(\tau \circ \tilde \alpha)^{r \to r}] \vee \tilde \alpha(r) = \tilde \alpha(r).
        \]
        \item \emph{$G$-equivariance}: Let $\rho \in \P\Omega_\bullet(X)$. Denote by $\tilde \alpha \vee \rho$ the path $s \mapsto \tilde \alpha(s) \vee \rho(s)$ in $\mathcal F_\bullet(X)$. Then,
        \begin{align*}
            \Theta_r(\tilde \alpha \vee \rho)(s) &= \big[(\tau \circ (\tilde \alpha \vee \rho))^{r \to s}\big] \vee (\tilde \alpha \vee \rho)(r) \\
            &= \big[(\tau \circ \tilde \alpha)^{r \to s}\big] \vee \tilde \alpha(r) \vee \rho(r) \\
            &= \Theta_r\tilde \alpha(s) \vee \rho(r).
        \end{align*}
        \item \emph{Reparameterization}: Let $f$ be a smooth, piecewise weakly monotone path in $\dom(\tilde \alpha)$. For each $u,s \in \dom(f)$, Lemma \ref{lem:ReparameterizationInvarianceOfRetrace} gives that
        \begin{align*}
            \Theta_u(\tilde \alpha \circ f)(s) &= \big[ (\tau \circ \tilde \alpha \circ f)^{u \to s} \big] \vee (\tilde \alpha \circ f)(u) \\
            &= \big[ (\tau \circ \tilde \alpha)^{f(u) \to f(s)} \big] \vee (\tilde \alpha \circ f)(u) \\
            &= \Theta_{f(u)}(\tilde \alpha)(f(s))
        \end{align*}
        \item \emph{Projection}: We have that, for each $s \in \dom(\tilde \alpha)$,
        \begin{align*}
            \Theta_u\big( \Theta_r\tilde \alpha \big)(s) &= \big[ (\tau \circ \Theta_r \tilde \alpha)^{u \to s} \big] \vee (\Theta_r \tilde \alpha)(u) \\
            &= \big[ (\tau \circ \tilde \alpha)^{u \to s} \big] \vee \big[ (\tau \circ \tilde \alpha)^{r \to u} \big] \vee \tilde \alpha(r) \\
            &= \big[ (\tau \circ \tilde \alpha)^{r \to s} \big] \vee \tilde \alpha(r) \\
            &= (\Theta_r \tilde \alpha)(s)
        \end{align*}
        where we have used the lifting property $\tau \circ \Theta_r \tilde \alpha = \tau \circ \tilde \alpha$.
    \end{enumerate}

\end{proof}

The universal connection $\Theta : \Pin \mathcal F_\bullet(X) \to \P \mathcal F_\bullet(X)$, being a diffeological connection, gives rise to a horizontal lifting function 
\[
    \hor_\Theta : \ev^*_X\big(\mathcal F_\bullet(X)\big) \to \P\mathcal F_\bullet(X),
\]
as described in Section \ref{subsec:ConnectionsAndHolonomyInDiffeology}. We write down an explicit formula for it:
\begin{proposition}
    \label{prop:HorizontalLiftOfUniversalConnection}
    Given $\alpha : (a, b) \to X$, an initial time $t_0 \in (a,b)$, and an initial path $[\beta_0] \in \tau\inverse(\alpha(t_0))$, define $\bar \alpha_{\beta_0} : (a, b) \to \mathcal F_\bullet(X)$ to be Singer's horizontal lift in \citep[\S3]{Singer}. Explicitly, $\bar \alpha_{\beta_0}$ is given by the formula
    \[
        \bar \alpha_{\beta_0}(s) = [\alpha^{t_0 \to s}] \vee [\beta_0].
    \]
    Then, we have $\hor_\Theta(\alpha, t_0, [\beta_0]) = \bar \alpha_{\beta_0}$.
\end{proposition}
\begin{proof}
    Since
    \[
        \hor_\Theta(\alpha, t_0, [\beta_0])(t_0) = [\beta_0] = \bar \alpha_{\beta_0}(t_0)
    \]
    and $\tau \circ \hor_\Theta(\alpha, t_0, [\beta_0]) = \alpha = \tau \circ \bar \alpha_{\beta_0}$, we need only show that $\bar \alpha_{\beta_0}$ is a $\Theta$-horizontal path in $\mathcal F_\bullet(X)$ (by Proposition \ref{prop:HorizontalPathsThatMeetCoincide}). Indeed, for each $s \in \dom(\alpha)$,
    \begin{align*}
        (\Theta_{t_0}\bar \alpha_{\beta_0})(s) &= \big[ (\tau \circ \bar \alpha_{\beta_0})^{t_0 \to s} \big] \vee \bar \alpha_{\beta_0}(t_0) \\
        &= \big[ \alpha^{t_0 \to s} \big] \vee [ \beta_0 ] \\
        &= \bar \alpha_{\beta_0}(s).
    \end{align*}
    Since $\hor_\Theta(\alpha, t_0, [\beta_0])$ and $\bar \alpha_{\beta_0}$ are horizontal paths which agree at a point (namely, $s = t_0$), they must coincide.

\end{proof}

\begin{remark}
    Proposition \ref{prop:HorizontalLiftOfUniversalConnection} says the lifting of Singer in \citep{Singer} is horizontal with respect to the diffeological connection $\Theta$ of Definition \ref{def:UniversalConnection}.
\end{remark}

\section{Applications of the Diffeological Universal Connection}
\label{sec:Kobayashi}

In this section, we apply the constructions of Section \ref{sec:PathSpaces} to state and prove Kobayashi's Theorem for Diffeology (Theorem \ref{thm:Kobayashi}) along with some related results. As in Section \ref{subsec:PathConcatenation}, whenever a representative $\alpha \in \smooth([0,1], X)$ of a class $[\alpha] \in \mathcal F(X)$, chosen constant near $0$ and $1$, is the input of a horizontal lifting function, it stands for its extension $\hat \alpha \in \P X$; for instance, $\hor_A(\alpha, 0, \xi)$ means $\hor_A(\hat \alpha, 0, \xi)$. If $f : X \to X'$ is smooth, then $f \circ \alpha$ is again constant near $0$ and $1$, and $\widehat{f \circ \alpha} = f \circ \hat \alpha$.

\subsection{Bundle Reconstruction from a Holonomy Representation}
\label{subsec:BundleReconstruction}

Let $(X, \bullet)$ be a path-connected, pointed diffeological space. In Section \ref{subsec:UniversalConnection}, we described the pointed path-bundle $\tau : \mathcal F_\bullet(X) \to X$ and the universal connection $\Theta : \Pin \mathcal F_\bullet(X) \to \P\mathcal F_\bullet(X)$. Suppose now that a smooth group homomorphism $H : \Omega_\bullet(X) \to G$ is given, where $G$ is a diffeological group. Then we have the associated bundle 
\[
    \tau_H : \mathcal F_\bullet(X) \times_H G \to X,\quad \text{where } \tau_H(\llbracket [\alpha], g \rrbracket) = \tau([\alpha]) = \alpha(1),
\]
which is a principal $G$-bundle (Definition \ref{def:AssociatedBundle}). The following two lemmas says that the universal connection $\Theta$ carries over to the associated bundle:
\begin{lemma}
    \label{lem:AssociatedUniversalConnectionIsConnection}
    Define $\Theta^H : \Pin(\mathcal F_\bullet(X) \times_H G) \to \P(\mathcal F_\bullet(X) \times_H G)$ by
    \[
        \Theta^H_r\llbracket \tilde \alpha, \rho\rrbracket(s) = \llbracket\Theta_r\tilde \alpha(s), \rho(r)\rrbracket,
    \]
    where $\llbracket\tilde \alpha, \rho\rrbracket$ denotes the path $s \mapsto \llbracket \tilde \alpha(s), \rho(s)\rrbracket$ in $\mathcal F_\bullet(X) \times_H G$ for paths $\tilde \alpha \in \P \mathcal F_\bullet(X)$ and $\rho \in \P G$. Then $\Theta^H$ is a diffeological connection on $\tau_H : \mathcal F_\bullet(X) \times_H G \to X$.
\end{lemma}
\begin{proof}
    Let $(\tilde{\tilde \alpha}, r) \in \Pin(\mathcal F_\bullet(X)) \times_H G)$. We show that $\tilde{\tilde \alpha}(s) = \llbracket \tilde \alpha(s), \rho(s)\rrbracket$ for paths $\tilde \alpha : \dom(\tilde{\tilde \alpha}) \to \mathcal F_\bullet(X)$ and $\rho : \dom(\tilde{\tilde \alpha}) \to G$. Say $\tilde{\tilde \alpha}(r) = \llbracket [\eta], g \rrbracket$. For $\alpha := \tau_H \circ \tilde{\tilde \alpha}$, the paths $\tilde{\tilde \alpha}$ and $\hor_{\Theta}(\alpha, r, [\eta])$ cover $\alpha$. By Lemma \ref{lem:CoveringPlotsDifferByGplot}, there is a $\rho_0 : \dom(\tilde{\tilde \alpha}) \to G$ for which
    \[
    \tilde{\tilde\alpha}(s) = \llbracket \hor_{\Theta}(\alpha, r, [\eta])(s), g \rho_0(s) \rrbracket,
    \]
    and so the claim follows by setting $\tilde \alpha = \hor_{\Theta}(\alpha, r, [\eta])$ and $\rho(s) = g \rho_0(s)$.

    We show now that $\Theta^H$ is well-defined. Indeed, if $\llbracket \tilde \alpha(r), \rho(r) \rrbracket = \llbracket [\eta], g \rrbracket$, then there exists $[\gamma] \in \Omega_\bullet(X)$ with $\tilde \alpha(r) = [\eta] \vee [\gamma]$ and $\rho(r) = H([\gamma])\inverse g$, so
    \begin{align*}
        \llbracket \Theta_r \tilde \alpha(s), \rho(r) \rrbracket &= \llbracket [\alpha^{r \to s}] \vee \tilde \alpha(r), \rho(r) \rrbracket \\
        &= \llbracket [\alpha^{r \to s}] \vee [\eta] \vee [\gamma], H([\gamma])\inverse g \rrbracket \\
        &= \llbracket [\alpha^{r \to s}] \vee [\eta], g \rrbracket.
    \end{align*}
    Since the right-hand side is formulated without any reference to $(\tilde \alpha, \rho)$, it follows that $\Theta^H$ is well-defined.

    It remains to show that $\Theta^H$ is a connection. We first show it is smooth. Let $u \mapsto (\tilde{\tilde \alpha}_u, r_u)$ be a plot for $\Pin(\mathcal F_\bullet(X) \times_H G)$. Since $u \mapsto \tilde{\tilde\alpha}_u$ is a plot for $\P(\mathcal F_\bullet(X) \times_H G)$, it follows that for any plot $u \mapsto s_u \in \dom(\tilde{\tilde \alpha}_u)$, the plots $u \mapsto \tilde{\tilde \alpha}_u(s_u)$ locally lifts into a plot $u \mapsto (\tilde \alpha_u, g_u)$ for $\mathcal F_\bullet(X) \times G$. The plot $u \mapsto \tilde \alpha_u$ of $\mathcal F_\bullet(X)$ again has a 2-step local lift to a plot $u \mapsto \alpha_u$ for $\smooth([0,1], X)$. By smoothness of $\Theta$, the composition
    \[
        u \mapsto \Theta_{r_u}^H(\llbracket \tilde \alpha_u, g_u \rrbracket)(s_u) = \llbracket \Theta_{r_u}\tilde \alpha_u(s_u), g_u \rrbracket
    \]
    is a plot, which shows smoothness of $\Theta^H$. The axioms of a connection for $\Theta^H$ now follow because $\Theta$ is a connection.

\end{proof}

\begin{lemma}
    \label{lem:AssociatedHorizontalLiftOfUniversalConnection}
    For a path $\alpha : (a, b) \to X$, an initial time $t_0 \in (a,b)$, and an initial point $\llbracket [\beta_0], g_0 \rrbracket \in \mathcal F_\bullet(X) \times_H G$, we have
    \[
        \hor_{\Theta^H}(\alpha, t_0, \llbracket [\beta_0], g_0 \rrbracket)(s) = \llbracket \bar \alpha_{\beta_0}(s), g_0 \rrbracket = \llbracket [\alpha^{t_0 \to s}] \vee [\beta_0], g_0 \rrbracket.
    \]
\end{lemma}
\begin{proof}
    Let $\bar{\bar \alpha}(s) = \llbracket [\alpha^{t_0 \to s}] \vee [\beta_0], g_0 \rrbracket$. Then, 
    \[
    (\tau_H \circ \bar{\bar \alpha})(s) = \tau([\alpha^{t_0 \to s}] \vee [\beta_0]) = \alpha(s) = \tau_H \circ (\hor_{\Theta^H}(\alpha, t_0, \llbracket [\beta_0], g_0 \rrbracket))(s).
    \]
    and $\bar{\bar \alpha}(t_0) = \llbracket [\beta_0], g_0 \rrbracket = \hor_{\Theta^H}(\alpha, t_0, \llbracket [\beta_0], g_0 \rrbracket)(t_0)$. Since $s \mapsto [\alpha^{t_0 \to s}] \vee [\beta_0]$ is $\Theta$-horizontal, it follows that $\bar{\bar \alpha}$ is $\Theta^H$-horizontal. By Proposition \ref{prop:HorizontalPathsThatMeetCoincide}, the claimed formula follows.

\end{proof}

\begin{theorem}
    [Diffeological Kobayashi Theorem]
    \label{thm:Kobayashi}
    Let $(X, \bullet)$ be a path-connected, pointed diffeological space. Given a smooth homomorphism $H : \Omega_\bullet(X) \to G$, where $G$ is a diffeological group, the associated bundle $\tau_H : \mathcal F_\bullet(X) \times_H G \to X$ is a diffeological principal $G$-bundle, $\Theta^H$ is a diffeological connection on it, and the holonomy representation of $\Theta^H$ at some point $\xi_\bullet^H$ is precisely $H$; that is, $H^{\Theta^H}_{\xi_\bullet^H} = H$.
\end{theorem}
\begin{proof}
    The desired lifting point in $\mathcal F_\bullet(X) \times_H G$ is 
    \[
        \xi_\bullet^H := \llbracket [c_\bullet], e \rrbracket,
    \]
    where $c_\bullet$ is the constant path at $\bullet$ and $e$ is the identity element of $G$. Let us compute the holonomy of $\Theta^H$ at $\xi_\bullet^H$. For $[\gamma] \in \Omega_\bullet(X)$,
    \begin{align*}
        \xi_\bullet^H \cdot H^{\Theta^H}_{\xi_\bullet^H}([\gamma]) &= \hor_{\Theta^H}(\gamma, 0, \xi_\bullet^H)(1) \\ 
        &= \llbracket\bar \gamma_\bullet(1), e\rrbracket \\
        &= \llbracket [\gamma], e\rrbracket \\
        &= \llbracket [\gamma] \vee [\gamma]\inverse, H([\gamma])\rrbracket \\
        &= \xi_\bullet^H\cdot H([\gamma]).
    \end{align*}
    Since $[\gamma] \in \Omega_\bullet(X)$ was arbitrary, we have $H^{\Theta^H}_{\xi_\bullet^H} = H$. 

\end{proof}

The following theorem roughly says the Singer's path-bundle construction is invertible, which also explains the choice of the word ``reconstruction.''

\begin{theorem}
    \label{thm:Reconstruction}
    Let $\pi : (E, \xi_\bullet) \to (X, \bullet)$ be a diffeological principal $G$-bundle. Given a connection $A : \Pin E \to \P E$, there is a $G$-equivariant diffeomorphism
    \[
        \Phi : \mathcal F_\bullet(X) \times_H G \to E,
    \]
    such that $\pi \circ \Phi = \tau_H$, where $H := H^A_{\xi_\bullet}$ is the holonomy representation of $A$. Moreover, the diffeomorphism $\Phi$ takes $\Theta^H$-horizontal paths to $A$-horizontal paths, in that the following diagram commutes:
    \[
    \begin{tikzcd}
        \Pin(\mathcal F_\bullet(X)\times_H G) \arrow{r} \arrow[swap]{d}{\Theta^H} & \Pin E \arrow{d}{A} \\
        \P(\mathcal F_\bullet(X) \times_H G) \arrow{r} & \P E
    \end{tikzcd}
    \]
    That is,
    \[
    \Phi \circ \Theta^H_r\llbracket \tilde \alpha, \rho\rrbracket = A_r(\Phi\circ \llbracket \tilde \alpha, \rho\rrbracket),
    \]
    for every initialized path $(\llbracket \tilde \alpha, \rho\rrbracket, r) \in \Pin(\mathcal F_\bullet(X) \times_H G)$.
\end{theorem}
\begin{proof}
    Define $\Phi : \mathcal F_\bullet(X) \times_H G \to E$ by
    \[
        \Phi(\llbracket [\alpha], g\rrbracket) = \hor_A(\alpha, 0, \xi_\bullet)(1)\cdot g.
    \]
    Then $\Phi$ is a well-defined map, for
    \begin{align*}
        \Phi(\llbracket [\alpha] \vee [\gamma], H([\gamma])\inverse g\rrbracket) &= \hor_A(\alpha \vee \gamma, 0, \xi_\bullet)(1)\cdot H([\gamma])\inverse g \\
        &= \big( \hor_A(\alpha, 0, \xi_\bullet H([\gamma])) \vee \hor_A(\gamma, 0, \xi_\bullet) \big)(1) \cdot H([\gamma])\inverse g \\
        &= \hor_A(\alpha, 0, \xi_\bullet)(1)H([\gamma])\cdot H([\gamma])\inverse g \\
        &= \Phi(\llbracket [\alpha], g\rrbracket).
    \end{align*}
    Now a plot $r \mapsto \llbracket [\alpha_r], g_r \rrbracket$ for $\mathcal F_\bullet(X)\times_H G$ locally lifts into a plot $r \mapsto ([\alpha_r], g_r)$ for $\mathcal F_\bullet(X) \times G$, and we see that $r \mapsto \Phi(\llbracket [\alpha_r], g_r\rrbracket) = \hor_A(\alpha_r, 0, \xi_\bullet)(1)g_r$ is smooth because $\hor_A$ is smooth. In other words: $\Phi : \mathcal F_\bullet(X) \times_H G \to E$ is smooth.

    The inverse of $\Phi$ is given as follows: for any $\xi \in E$, let $([\alpha], g) \in \mathcal F_\bullet(X) \times G$ be any pair satisfying $\xi = \hor_A(\alpha, 0, \xi_\bullet)(1) g$. Suppose $([\beta], h) \in \mathcal F_\bullet(X) \times G$ is another pair with $\xi = \hor_A(\beta, 0, \xi_\bullet)(1) h$. Then $[\beta^\leftarrow] \vee [\alpha]$ is a loop with
    \begin{align*}
    \hor_A(\beta^\leftarrow \vee \alpha, 0, \xi_\bullet)(1) &= \hor_A(\beta^\leftarrow, 0, \xi \cdot h\inverse h \cdot g\inverse)(1) \\
    &= \hor_A(\beta^\leftarrow, 0, \xi h\inverse )(1) \cdot h g\inverse \\
    &= \big( \hor_A(\beta, 0, \xi_\bullet ) \big)^\leftarrow(1)\cdot hg\inverse \\
    &= \hor_A(\beta, 0, \xi_\bullet )(0)\cdot h g\inverse \\
    &= \xi_\bullet \cdot hg\inverse,
    \end{align*}
    where we have used Proposition \ref{prop:PropertiesOfHorizontalLift}(iii) and (iv). Therefore, we have $H([\beta^\leftarrow \vee \alpha]) = hg\inverse$. All this is to say that
    \begin{align*}
        \llbracket [\beta], h\rrbracket &= \llbracket [\beta] \vee [\beta^\leftarrow \vee \alpha], H([\beta^\leftarrow \vee \alpha])\inverse h\rrbracket \\
        &= \llbracket [\alpha], (hg\inverse)\inverse h\rrbracket\\
        &= \llbracket [\alpha], g\rrbracket,
    \end{align*}
    Put $\Phi\inverse(\xi) = \llbracket [\alpha], g\rrbracket$ for some pair $([\alpha], g) \in \mathcal F_\bullet(X) \times G$ satisfying $\xi = \hor_A(\alpha, 0, \xi_\bullet)(1) g$, which is well-defined by the discussion just now. Let us show that $\Phi\inverse : E \to \mathcal F_\bullet(X) \times_H G$ is a smooth map; fix a plot $r \mapsto \xi_r$ for $E$. As in the proof of Proposition \ref{prop:FPrincipalBundle}, we can choose $[\alpha_r] \in \mathcal F_\bullet(X)$ with $\alpha_r(1) = \pi(\xi_r)$ depending smoothly on $r$. Since $\hor_A$ is smooth, there exists (by Lemma \ref{lem:CoveringPlotsDifferByGplot}) a plot $r \mapsto g_r$ for $G$ such that $\hor_A(\alpha_r, 0, \xi_\bullet)(1)g_r = \xi_r$. But this implies $r \mapsto \Phi\inverse(\xi_r) = \llbracket [\alpha_r], g_r\rrbracket$ is smooth. Therefore, $\Phi : \mathcal F_\bullet(X) \times_H G \to E$ is diffeomorphism.

    Let us now show that $\Phi$ takes the universal connection $\Theta^H$ on $\tau_H : \mathcal F_\bullet(X) \times_H G \to X$ to the connection $A$ on $\pi : E \to X$. Indeed,
    \begin{alignat*}{2}
        &A_r(\Phi\circ \llbracket \tilde \alpha, \rho\rrbracket)(s) \\
        &\quad = \hor_A\big(\pi \circ \Phi \circ \llbracket \tilde \alpha, \rho\rrbracket, r, \Phi(\llbracket \tilde \alpha(r), \rho(r)\rrbracket)\big)(s) \\
        &\quad = \hor_A \big( \tau_H\circ \llbracket \tilde \alpha, \rho\rrbracket, r, \hor_A(\tilde \alpha(r), 0, \xi_\bullet)(1)\big)(s) \cdot \rho(r), &\quad& \text{(by Proposition \ref{prop:PropertiesOfHorizontalLift}(i))} \\
        &\quad \mathrlap{{}= \hor_A\big( \tau \circ \tilde \alpha, \ell_{r,s}(0), \hor_A(\tilde \alpha(r), 0, \xi_\bullet)(1) \big)(\ell_{r,s}(1)) \cdot \rho(r)} \\
        &\quad = \hor_A\big((\tau \circ \tilde \alpha) \circ \ell_{r,s}, 0, \hor_A(\tilde \alpha(r), 0, \xi_\bullet)(1)\big)(1) \cdot \rho(r) &\quad& \text{(by Proposition \ref{prop:PropertiesOfHorizontalLift}(ii))} \\
        &\quad = \hor_A\big( (\tau \circ \tilde \alpha)^{r \to s} \vee \tilde \alpha(r), 0, \xi_\bullet \big)(1) \cdot \rho(r) &\quad& \text{(by Proposition \ref{prop:PropertiesOfHorizontalLift}(iii))} \\
        &\quad = \Phi\big( (\tau \circ \tilde \alpha)^{r \to s} \vee \tilde \alpha(r), \rho(r) \big) \\
        &\quad = \Phi\big( \Theta^H_r \llbracket \tilde \alpha, \rho \rrbracket(s) \big),
    \end{alignat*}
    which was the claim.

\end{proof}

\begin{corollary}
    \label{cor:Classification}
    Let $\pi : E \to X$ and $\pi' : E' \to X$ be principal $G$-bundles with connections $A$ and $A'$. Suppose that there are points $\xi \in \pi\inverse(\bullet)$ and $\xi' \in (\pi')\inverse(\bullet)$ such that the holonomy maps $H := H_{\xi}^{A}, H' := H^{A'}_{\xi'} : \Omega_\bullet(X) \to G$ are conjugate, i.e. there is a $g \in G$ such that $H'([\gamma]) = g H([\gamma])g\inverse$ for all $[\gamma] \in \Omega_\bullet(X)$. Then there is a $G$-bundle isomorphism $\Psi : E \to E'$ for which 
    \[
        \Psi \circ A_r \tilde \alpha = A'_r(\Psi\circ \tilde \alpha), \quad \text{for all } (\tilde \alpha, r) \in \Pin E.
    \]
\end{corollary}
\begin{proof}
    Define a mapping $f : \mathcal F_\bullet(X) \times_{H} G \to \mathcal F_\bullet(X) \times_{H'} G$ by
    \[
    f(\llbracket [\alpha], h\rrbracket_H) = \llbracket [\alpha], g h\rrbracket_{H'},
    \]
    where $g \in G$ was the fixed group element such that $H' = gH g\inverse$ or, equivalently, $g H\inverse = (H')\inverse g$. Then $f$ is well-defined, since
    \begin{align*}
        f(\llbracket [\alpha] \vee [\gamma], H([\gamma])\inverse h\rrbracket_H) &= \llbracket [\alpha] \vee [\gamma], g H([\gamma])\inverse h\rrbracket_{H'} \\
        &= \llbracket [\alpha] \vee [\gamma], H'([\gamma])\inverse g h\rrbracket_{H'} \\
        &= \llbracket [\alpha], g h \rrbracket_{H'} \\
        &= f(\llbracket [\alpha], h\rrbracket_H).
    \end{align*}

    It is clear that $f$ is a $G$-equivariant diffeomorphism with $f\inverse(\llbracket [\alpha], h\rrbracket_{H'}) = \llbracket [\alpha], g\inverse h\rrbracket_H$. Observe also that
    \[
        f \circ \Theta^{H} = \Theta^{H'} \circ (f \times \mathrm{id}),
    \]
    where $\Theta^{H}$ resp. $\Theta^{H'}$ is the (induced) universal connection on $\tau_{H} : \mathcal F_\bullet(X) \times_{H} G \to X$ resp. $\tau_{H'} : \mathcal F_\bullet(X) \times_{H'} G \to X$. By Theorem \ref{thm:Reconstruction}, there exists a $G$-equivariant diffeomorphism $\Phi : E \to \mathcal F_\bullet(X) \times_{H} G$ such that
    \[
        \Theta^{H} = \Phi\inverse \circ A \circ (\Phi \times \mathrm{id})
    \]
    and, similarly, we have $\Phi' : E' \to \mathcal F_\bullet(X) \times_{H'} G$. Then $\Psi := \Phi' \circ f \circ \Phi\inverse$ defines a $G$-equivariant diffeomorphism $E \to E'$ such that
    \begin{align*}
        A &= \Phi \circ \Theta^{H} \circ (\Phi\inverse \times \mathrm{id}) \\
        &= \Phi \circ f\inverse \circ \Theta^{H'} \circ (f \times \mathrm{id}) \circ (\Phi\inverse \times \mathrm{id}) \\
        &= \Phi \circ f\inverse \circ (\Phi')\inverse \circ A' \circ (\Phi' \times \mathrm{id}) \circ (f \times \mathrm{id}) \circ (\Phi\inverse \times \mathrm{id}) \\
        &= \Psi\inverse \circ A' \circ (\Psi \times \mathrm{id}),
    \end{align*}
    where we used that $\Theta^{H'} = (\Phi')\inverse \circ A' \circ (\Phi' \times \id)$.

\end{proof}

\subsection{A Categorical Interpretation of the Bundle Reconstruction}
\label{subsec:CategoryEquivalence}

We conclude by organizing the results of this section in category-theoretic language. Throughout, $G$ is a fixed diffeological group. We define a \emph{holonomy category} $\Holonomy^G$ and a \emph{bundle-connection category} $\BundleConnection^G$, together with a reconstruction functor $\mathcal K : \Holonomy^G \to \BundleConnection^G$ and a holonomy functor $\mathcal H : \BundleConnection^G \to \Holonomy^G$. Theorem \ref{thm:Kobayashi} is essentially the statement that $\mathcal H \circ \mathcal K$ is the identity functor, and Theorem \ref{thm:Reconstruction} provides, object by object, an isomorphism from $\mathcal K \circ \mathcal H$ to the identity functor. The remaining point is that these isomorphisms are natural (Theorem \ref{thm:CategoryEquivalence}). We first record two lemmas.
\begin{lemma}
    \label{lem:Pushforward}
    Let $f : X \to X'$ be a smooth map of diffeological spaces. Then
    \[
        f_* : \mathcal F(X) \to \mathcal F(X'), \quad f_*[\alpha] = [f \circ \alpha],
    \]
    is a well-defined smooth map. Moreover:
    \begin{enumerate}[(i)]
        \item $f_*$ is compatible with endpoints, concatenation and constant paths: that is, $f_*[\alpha]$ is a class of paths from $f(\alpha(0))$ to $f(\alpha(1))$, $f_*([\beta] \vee [\alpha]) = f_*[\beta] \vee f_*[\alpha]$ whenever $\beta(0) = \alpha(1)$, and $f_*[c_x] = [c_{f(x)}]$ for every $x \in X$.
        \item For each $x_0 \in X$, the restriction $f_* \upharpoonright \Omega_{x_0}(X)$ gives a smooth group homomorphism $\Omega_{x_0}(X) \to \Omega_{f(x_0)}(X')$.
        \item $(\id_X)_* = \id_{\mathcal F(X)}$ and $(f' \circ f)_* = f'_* \circ f_*$ for every smooth map $f' : X' \to X''$.
    \end{enumerate}
\end{lemma}
\begin{proof}
    If $\eta_i = \eta_{i-1} \circ \phi$ for some $\phi \in \mathcal M$, then $f \circ \eta_i = (f \circ \eta_{i-1}) \circ \phi$. Hence post-composition with $f$ induces a map $\smooth_\sim([0,1], X) \to \smooth_\sim([0,1], X')$, $\corners{\alpha} \mapsto \corners{f \circ \alpha}$. Post-composition with $f$ also takes paths which are constant near $0$ and $1$ to paths which are constant near $0$ and $1$, and satisfies
    \[
    f \circ (\beta \vee \alpha) = (f \circ \beta) \vee (f \circ \alpha), \qquad f \circ \eta^\leftarrow = (f \circ \eta)^\leftarrow, \qquad f \circ c_x = c_{f(x)}.
    \]
    Consequently, $\corners{\alpha} \approx \corners{\beta}$ implies $\corners{f \circ \alpha} \approx \corners{f \circ \beta}$, so $f_*$ is well-defined, and in fact (i) holds.

    For smoothness, let $r \mapsto [\alpha_r]$ be a plot for $\mathcal F(X)$. By definition of the quotient diffeology, it lifts locally (again in two steps) to a plot $r \mapsto \alpha_r$ for $\smooth([0,1], X)$. By Lemma \ref{lem:CompositionIsSmooth}, $r \mapsto f \circ \alpha_r$ is a plot for $\smooth([0,1], X')$, so $r \mapsto f_*[\alpha_r] = [f \circ \alpha_r]$ is locally the image of a plot under the two quotient maps, hence is a plot for $\mathcal F(X')$. The remaining items are immediate.

\end{proof}

By a \textbf{$G$-bundle morphism} $(F, f) : \pi \to \pi'$ between diffeological principal $G$-bundles $\pi : E \to X$ and $\pi' : E' \to X'$, we mean a pair of smooth maps $F : E \to E'$ and $f : X \to X'$ such that $\pi' \circ F = f \circ \pi$ and $F(\xi g) = F(\xi) g$ for all $\xi \in E$ and $g \in G$. 
\begin{lemma}
    \label{lem:ConnectionPreservingMorphism}
    Let $\pi : E \to X$ and $\pi' : E' \to X'$ be diffeological principal $G$-bundles equipped with connections $A$ and $A'$, respectively, and let $(F, f) : \pi \to \pi'$ be a $G$-bundle morphism. Then, 
    \[
    F \circ A_r\tilde \alpha = A'_r(F \circ \tilde \alpha) \quad \text{for all } (\tilde \alpha, r) \in \Pin E
    \]
    if and only if 
    \[
    F \circ \hor_A(\alpha, t_0, \xi) = \hor_{A'}(f \circ \alpha, t_0, F(\xi)) \quad \text{for all } (\alpha, t_0, \xi) \in \ev_X^*(E).
    \]
\end{lemma}
\begin{proof}
    Assume first that $F \circ A_r\tilde \alpha = A'_r(F \circ \tilde \alpha)$ for all $(\tilde \alpha, r) \in \Pin E$. Write $\bar \alpha$ for $\hor_A(\alpha, t_0, \xi)$. Then, since $A_{t_0}\bar \alpha = \hor_A(\alpha, t_0, \xi)$ and $\pi' \circ F = f \circ \pi$, one has
    \begin{align*}
    F \circ \hor_A(\alpha, t_0, \xi) &= F \circ A_{t_0} \bar \alpha  = A_{t_0}'(F \circ \bar \alpha) \\
    &= \hor_{A'}(\pi' \circ F \circ \bar \alpha, t_0, (F \circ \bar \alpha)(t_0)) = \hor_{A'}(f\circ \alpha, t_0, F(\xi)).
    \end{align*}

    Conversely, assume $F \circ \hor_A(\alpha, t_0, \xi) = \hor_{A'}(f \circ \alpha, t_0, F(\xi))$ for all $(\alpha, t_0, \xi) \in \ev_X^*(E)$ and let $(\tilde \alpha, r) \in \Pin E$. Write $\alpha = \pi \circ \tilde \alpha$. We then have $A_r \tilde \alpha = \hor_A(\alpha, r, \tilde \alpha(r))$, so
    \begin{align*}
    F \circ A_r \tilde \alpha &= F \circ \hor_A(\alpha, r, \tilde \alpha(r)) = \hor_{A'}(f \circ \alpha, r, F(\tilde \alpha(r))) \\
    &= \hor_{A'}(\pi' \circ F \circ \tilde \alpha, r, F(\tilde \alpha(r))) = A'_r(F \circ \tilde \alpha).
    \end{align*}
\end{proof}

We now formally define the two categories mentioned in the introduction.
\begin{definition}
    \label{def:HolonomyCategory}
    The \textbf{holonomy category} is the category {\upshape $\Holonomy^G$} defined by the following data:
    \begin{enumerate}[(i)]
        \item An object is a triple $(X, x_0, H)$, where $(X, x_0)$ is a pointed, path-connected diffeological space and $H : \Omega_{x_0}(X) \to G$ is a smooth group homomorphism.
        \item A morphism $(f, g) : (X, x_0, H) \to (X', x_0', H')$ is a pair consisting of a pointed smooth map $f : (X, x_0) \to (X', x_0')$ and an element $g \in G$ such that
        \[
            H'(f_*[\gamma]) = g H([\gamma]) g\inverse \quad \text{for all } [\gamma] \in \Omega_{x_0}(X),
        \]
        where $f_*[\gamma] = [f \circ \gamma] \in \Omega_{x_0'}(X')$.
        \item {\thickmuskip=5mu plus 5mu minus 1.5mu
        The composite of $(f, g) : (X, x_0, H) \to (X', x_0', H')$ and $(f', g') : (X', x_0', H') \to (X'', x_0'', H'')$ is $(f' \circ f, g'g)$, and the identity morphism of $(X, x_0, H)$ is $(\id_X, e)$.}
    \end{enumerate}
\end{definition}
The above data really does form a category. Indeed, for morphisms $(f, g)$ and $(f', g')$ as in item (iii) above, Lemma \ref{lem:Pushforward}(iii) gives, for all $[\gamma] \in \Omega_{x_0}(X)$,
\[
    H''\big((f' \circ f)_*[\gamma]\big) = H''\big(f'_*(f_*[\gamma])\big) = g' H'(f_*[\gamma]) (g')\inverse = (g'g) H([\gamma]) (g'g)\inverse,
\]
so $(f' \circ f, g'g)$ is again a morphism. Associativity and the unit laws follow from those for composition of maps and for multiplication in $G$.

\begin{definition}
    \label{def:BundleConnectionCategory}
    The \textbf{bundle-connection category} is the category $\BundleConnection^G$ of pointed diffeological principal $G$-bundles equipped with a connection, defined by the following data:
    \begin{enumerate}[(i)]
        \item An object is a pair $(\pi : (E, \xi_0) \to (X, x_0), A)$, where $\pi : E \to X$ is a diffeological principal $G$-bundle over a path-connected diffeological space $X$ and $A : \Pin E \to \P E$ is a diffeological connection. When the spaces and base points are understood, we abbreviate this object to $(\pi, A)$.
        \item A morphism $(F, f) : (\pi, A) \to (\pi', A')$ from $(\pi : (E, \xi_0) \to (X, x_0), A)$ to $(\pi' : (E', \xi_0') \to (X', x_0'), A')$ is a $G$-bundle morphism $(F, f) : \pi \to \pi'$ such that $f(x_0) = x_0'$ and
        \[
        F \circ A_r\tilde \alpha = A'_r(F \circ \tilde \alpha) \quad \text{for all } (\tilde \alpha, r) \in \Pin E.
        \]
        Visually, the following two squares commute:
        \[
        \begin{tikzcd}
            E \arrow{r}{F} \arrow[swap]{d}{\pi} & E' \arrow{d}{\pi'} & & \Pin E \arrow{r}{\Pin F} \arrow[swap]{d}{A} & \Pin E' \arrow{d}{A'} \\
            (X, x_0) \arrow[swap]{r}{f} & (X', x_0') & & \P E \arrow[swap]{r}{\P F} & \P E'
        \end{tikzcd}
        \]
        where $\P F(\tilde \alpha) = F \circ \tilde \alpha$ and $\Pin F(\tilde \alpha, r) = (F \circ \tilde \alpha, r)$.
        \item Composition is given by $(F', f') \circ (F, f) = (F' \circ F, f' \circ f)$, and the identity morphism of $(\pi, A)$ is $(\id_E, \id_X)$.
    \end{enumerate}
\end{definition}
The composite of two morphisms is again a morphism: $(F' \circ F, f' \circ f)$ is a $G$-bundle morphism with $f'(f(x_0)) = x_0''$, and $F' \circ F \circ A_r\tilde \alpha = F' \circ A'_r(F \circ \tilde \alpha) = A''_r(F' \circ F \circ \tilde \alpha)$. Note that a morphism $(F, f)$ is explicitly \emph{not} required to satisfy $F(\xi_0) = \xi_0'$. Since $\pi'(F(\xi_0)) = f(x_0) = x_0'$, there is a unique element $k_F \in G$ such that
\[
    F(\xi_0) = \xi_0' \cdot k_F.
\]
This element is what corresponds to the conjugating element of a morphism of $\Holonomy^G$:
\begin{lemma}
    \label{lem:HolonomyOfMorphism}
    Let $(F, f)$ be a {\upshape$\BundleConnection^G$}-morphism from $(\pi : (E, \xi_0) \to (X, x_0), A)$ to $(\pi' : (E', \xi_0') \to (X', x_0'), A')$. Then
    \[
        H^{A'}_{\xi_0'}(f_*[\gamma]) = k_F\, H^A_{\xi_0}([\gamma])\, k_F\inverse \quad \text{for all } [\gamma] \in \Omega_{x_0}(X).
    \]
    In other words, $(f, k_F)$ is a {\upshape $\Holonomy^G$}-morphism from $(X, x_0, H^A_{\xi_0})$ to $(X', x_0', H^{A'}_{\xi_0'})$.
\end{lemma}
\begin{proof}
    Write $H = H^A_{\xi_0}$, $H' = H^{A'}_{\xi_0'}$ and $k = k_F$. Let $[\gamma] \in \Omega_{x_0}(X)$. Using the $G$-equivariance of $F$, Lemma \ref{lem:ConnectionPreservingMorphism}, and the $G$-equivariance of $\hor_{A'}$ (Proposition \ref{prop:PropertiesOfHorizontalLift}(i)), we obtain
    \begin{align*}
        \xi_0' \cdot k H([\gamma]) &= F(\xi_0) \cdot H([\gamma]) \\
        &= F\big(\xi_0 \cdot H([\gamma])\big) \\
        &= F\big(\hor_A( \gamma, 0, \xi_0)(1)\big) \\
        &= \hor_{A'}\big(f \circ \gamma, 0, F(\xi_0)\big)(1) \\
        &= \hor_{A'}\big(f \circ \gamma, 0, \xi_0'\big)(1) \cdot k \\
        &= \xi_0' \cdot H'(f_*[\gamma]) k.
    \end{align*}
    It follows that $k H([\gamma]) = H'(f_*[\gamma]) k$, which is the claim.

\end{proof}

For a $\Holonomy^G$-object $(X, x_0, H)$, we write $\xi_0^H$ for the point $\llbracket [c_{x_0}], e \rrbracket_H \in \mathcal F_{x_0}(X) \times_H G$.
\begin{lemma}
    \label{lem:InducedBundleMap}
    Let $(f, g)$ be a morphism of the holonomy category from $(X, x_0, H)$ to $(X', x_0', H')$. Then the map
    \[
    F_{f,g} : \mathcal F_{x_0}(X) \times_H G \to \mathcal F_{x_0'}(X') \times_{H'} G, \quad F_{f,g}\big(\llbracket [\alpha], h \rrbracket_H\big) = \llbracket f_*[\alpha], gh \rrbracket_{H'},
    \]
    is well-defined, and $(F_{f,g}, f)$ is a $G$-bundle morphism from $\tau_H$ to $\tau_{H'}$ such that $F_{f,g}(\xi_0^H) = \xi_0^{H'} g$ and
    \[
    F_{f,g} \circ \Theta^H_r\tilde \alpha = \Theta^{H'}_r(F_{f,g} \circ \tilde \alpha) \quad \text{for all } (\tilde \alpha, r) \in \Pin(\mathcal F_{x_0}(X) \times_H G).
    \]
\end{lemma}
\begin{proof}
    We first show that $F_{f,g}$ is well-defined. Let $[\gamma] \in \Omega_{x_0}(X)$. Since $g H([\gamma])\inverse = H'(f_*[\gamma])\inverse g$, we have
    \begin{align*}
        F_{f,g}\big(\llbracket [\alpha] \vee [\gamma], H([\gamma])\inverse h \rrbracket_H\big) &= \llbracket f_*[\alpha] \vee f_*[\gamma], g H([\gamma])\inverse h \rrbracket_{H'} \\
        &= \llbracket f_*[\alpha] \vee f_*[\gamma], H'(f_*[\gamma])\inverse gh \rrbracket_{H'} \\
        &= \llbracket f_*[\alpha], gh \rrbracket_{H'} \\
        &= F_{f,g}\big(\llbracket [\alpha], h \rrbracket_H\big).
    \end{align*}

    We show now that $F_{f,g}$ is smooth. Let $r \mapsto \llbracket [\alpha_r], h_r \rrbracket_H$ be a plot for $\mathcal F_{x_0}(X) \times_H G$. By definition of the quotient diffeology, it lifts locally to a plot $r \mapsto ([\alpha_r], h_r)$ for $\mathcal F_{x_0}(X) \times G$. Since $f_* : \mathcal F_{x_0}(X) \to \mathcal F_{x_0'}(X')$ is smooth (Lemma \ref{lem:Pushforward}), the map $r \mapsto (f_*[\alpha_r], gh_r)$ is a plot for $\mathcal F_{x_0'}(X') \times G$. Composing with the quotient map shows that $r \mapsto F_{f,g}(\llbracket [\alpha_r], h_r \rrbracket_H)$ is locally a plot, hence a plot. Thus $F_{f,g}$ is smooth.

    Furthermore, for $h' \in G$, we have $F_{f,g}(\llbracket [\alpha], h \rrbracket_H \cdot h') = \llbracket f_*[\alpha], ghh' \rrbracket_{H'} = F_{f,g}(\llbracket [\alpha], h \rrbracket_H) \cdot h'$, and
    \[
    \tau_{H'}\big(F_{f,g}(\llbracket [\alpha], h \rrbracket_H)\big) = (f \circ \alpha)(1) = f\big(\tau_H(\llbracket [\alpha], h \rrbracket_H)\big).
    \]
    This means $(F_{f,g}, f)$ is a $G$-bundle morphism $\tau_H \to \tau_{H'}$. The fact that $f \circ c_{x_0} = c_{x_0'}$ gives $F_{f,g}(\xi_0^H) = \llbracket [c_{x_0'}], g \rrbracket_{H'} = \xi_0^{H'} g$. 

    Finally, we show that $F_{f,g}$ takes $\Theta^H$-horizontal paths to $\Theta^{H'}$-horizontal paths. Let $(\alpha, t_0, \llbracket [\beta_0], h_0 \rrbracket_H) \in \ev_X^*(\mathcal F_{x_0}(X) \times_H G)$, so that $\beta_0(1) = \alpha(t_0)$. Then, for every $s \in \dom(\alpha)$,
    \begin{align*}
        F_{f,g}\big(\hor_{\Theta^H}(\alpha, t_0, \llbracket [\beta_0], h_0 \rrbracket_H)(s)\big) &= F_{f,g}\big(\llbracket [\alpha^{t_0 \to s}] \vee [\beta_0], h_0 \rrbracket_H\big) \\
        &= \llbracket [(f \circ \alpha)^{t_0 \to s}] \vee [f \circ \beta_0], gh_0 \rrbracket_{H'} \\
        &= \hor_{\Theta^{H'}}\big(f \circ \alpha, t_0, \llbracket [f \circ \beta_0], gh_0 \rrbracket_{H'}\big)(s) \\
        &= \hor_{\Theta^{H'}}\big(f \circ \alpha, t_0, F_{f,g}(\llbracket [\beta_0], h_0 \rrbracket_H)\big)(s).
    \end{align*}
    Hence, by Lemma \ref{lem:ConnectionPreservingMorphism}, $F_{f,g} \circ \Theta^H_r\tilde \alpha = \Theta^{H'}_r(F_{f,g} \circ \tilde \alpha)$ for all $(\tilde \alpha, r)$
    
\end{proof}

We are now in a position to define the two functors announced at the beginning of this subsection. The \textbf{reconstruction functor} $\mathcal K : \Holonomy^G \to \BundleConnection^G$ is given by
\[
    \mathcal K(X, x_0, H) = \big(\tau_H : (\mathcal F_{x_0}(X) \times_H G, \xi_0^h) \to (X, x_0), \Theta^H\big), \qquad \mathcal K(f, g) = (F_{f,g}, f),
\]
and the \textbf{holonomy functor} $\mathcal H : \BundleConnection^G \to \Holonomy^G$ is given by
\[
    \mathcal H\big(\pi : (E, \xi_0) \to (X, x_0), A\big) = (X, x_0, H^A_{\xi_0}), \qquad \mathcal H(F, f) = (f, k_F).
\]
\begin{proposition}
    \label{prop:Functors}
    The assignments $\mathcal K$ and $\mathcal H$ are well-defined functors. Moreover, the composition $\mathcal H \mathcal K$ is the identity functor of the {\upshape$\Holonomy^G$}.
\end{proposition}
\begin{proof}
    For a $\Holonomy^G$-object $(X, x_0, H)$, Theorem \ref{thm:Kobayashi} says that $\mathcal K(X, x_0, H)$ is a $\BundleConnection^G$-object. For each $\Holonomy^G$-morphism $(f, g)$, Lemma \ref{lem:InducedBundleMap} implies that $\mathcal K(f,g)$ is a $\BundleConnection^G$-morphism. Suppose that $(f,g)$ and $(f',g')$ are composable $\Holonomy^G$-morphisms. Then, by Lemma \ref{lem:Pushforward}(iii), it follows that $F_{f',g'} \circ F_{f,g} = F_{f'\circ f, g' g}$. All this is to say that $\mathcal K$ is indeed a functor.

    For a $\BundleConnection^G$-object $(\pi : (E, \xi_0) \to (X, x_0), A)$, the holonomy representation $H^A_{\xi_0} : \Omega_{x_0}(X) \to G$ is a smooth group homomorphism (Proposition \ref{prop:HolonomyOnLoopGroupIsSmooth}), so $\mathcal H(\pi, A)$ is an object of $\Holonomy^G$. Lemma \ref{lem:HolonomyOfMorphism} says that $\mathcal H(F, f)$ is always a $\Holonomy^G$-morphism for $\BundleConnection^G$-morphisms $(F,f)$. For composable $\BundleConnection^G$-morphisms $(F, f) : (\pi, A) \to (\pi', A')$ and $(F', f') : (\pi', A') \to (\pi'', A'')$, the $G$-equivariance of $F'$ gives
    \[
        (F' \circ F)(\xi_0) = F'(\xi_0' \cdot k_F) = F'(\xi_0') \cdot k_F = \xi_0'' \cdot k_{F'} k_F,
    \]
    so $k_{F' \circ F} = k_{F'} k_F$, i.e.\ $\mathcal H(F' \circ F, f' \circ f) = \mathcal H(F', f') \circ \mathcal H(F, f)$.

    Now, on objects, Kobayashi's Theorem for Diffeology (Theorem \ref{thm:Kobayashi}) implies that 
    \[
    \mathcal H\mathcal K(X, x_0, H) = (X, x_0, H^{\Theta^H}_{\xi_0^H}) = (X, x_0, H).
    \]
    On morphisms, $F_{f,g}(\xi_0^H) = \xi_0^{H'} g$ means $k_{F_{f,g}} = g$, so $\mathcal H\mathcal K(f, g) = (f, k_{F_{f,g}}) = (f, g)$, and $\mathcal H\mathcal K$ is the identity.

\end{proof}

\begin{theorem}
    \label{thm:CategoryEquivalence}
    There is an equivalence of categories \upshape
    \[
    \Holonomy^G \simeq \BundleConnection^G.
    \]
    \itshape Moreover, $\mathcal H \circ \mathcal K$ is the identity functor, and the isomorphisms $\mathcal K\mathcal H(\pi, A) \to (\pi, A)$ provided by Theorem \ref{thm:Reconstruction} are the components of a natural isomorphism $\Phi : \mathcal K \circ \mathcal H \Rightarrow \mathrm{Id}$. In fact, $\mathcal K$ is left adjoint to $\mathcal H$, with unit the identity and counit $\Phi$, and this adjunction is an adjoint equivalence.
\end{theorem}
\begin{proof}
    We already have that $\mathcal K$ and $\mathcal H$ are functors with $\mathcal H \circ \mathcal K = \mathrm{Id}_{\Holonomy^G}$, by Proposition \ref{prop:Functors}. It remains to construct the natural isomorphism $\Phi$.

    Let $(\pi, A) = (\pi : (E, \xi_0) \to (X, x_0), A)$ be a $\BundleConnection^G$-object. Put $H = H^A_{\xi_0}$, so that $\mathcal K\mathcal H(\pi, A) = (\tau_H, \Theta^H)$. Let
    \[
        \Phi_{(\pi, A)} : \mathcal F_{x_0}(X) \times_H G \to E, \quad \Phi_{(\pi, A)}\big(\llbracket [\alpha], h \rrbracket_H\big) = \hor_A(\alpha, 0, \xi_0)(1) \cdot h,
    \]
    be the map of Theorem \ref{thm:Reconstruction}; by the same theorem, $(\Phi_{(\pi, A)}, \id_X)$ is a $\BundleConnection^G$-isomorphism $\mathcal K\mathcal H(\pi, A) \to (\pi, A)$. We write $\Phi_{(\pi, A)}$ also for $(\Phi_{(\pi, A)}, \id_X)$. Observe that $\Phi_{(\pi, A)}$ maps base point to base point:
    \[
    \Phi_{(\pi, A)}(\xi_0^H) = \hor_A(\hat c_{x_0}, 0, \xi_0)(1) = \xi_0.
    \]

    We show that the $\Phi_{(\pi, A)}$ are natural in $\BundleConnection^G$-morphisms. Let $(F, f) : (\pi, A) \to (\pi', A')$ be a $\BundleConnection^G$-morphism, where $(\pi', A') = (\pi' : (E', \xi_0') \to (X', x_0'), A')$. Put $H = H^A_{\xi_0}$ and $H' = H^{A'}_{\xi_0'}$, and write $k = k_F$, so that $F(\xi_0) = \xi_0' \cdot k$ and $\mathcal K\mathcal H(F, f) = \mathcal K(f, k) = (F_{f,k}, f)$. We show that the square
    \[
    \begin{tikzcd}[column sep = large]
        \mathcal K\mathcal H(\pi, A) \arrow[r, "{(F_{f,k},\, f)}"] \arrow[d, swap, "{(\Phi_{(\pi, A)},\, \id_X)}"] & \mathcal K\mathcal H(\pi', A') \arrow[d, "{(\Phi_{(\pi', A')},\, \id_{X'})}"] \\
        (\pi, A) \arrow[r, swap, "{(F,\, f)}"] & (\pi', A')
    \end{tikzcd}
    \]
    commutes. Both composites have base component $f$, so we simply show that $F \circ \Phi_{(\pi, A)} = \Phi_{(\pi', A')} \circ F_{f,k}$. Let $\llbracket [\alpha], h \rrbracket_H \in \mathcal F_{x_0}(X) \times_H G$. Using Lemma \ref{lem:ConnectionPreservingMorphism}, we compute:
    \begin{align*}
        F\Big(\Phi_{(\pi, A)}\big(\llbracket [\alpha], h \rrbracket_H\big)\Big) &= F\big(\hor_A(\alpha, 0, \xi_0)(1)\big) \cdot h \\
        &= \hor_{A'}\big(f \circ \alpha, 0, F(\xi_0)\big)(1) \cdot h \\
        &= \hor_{A'}\big(f \circ \alpha, 0, \xi_0' \cdot k\big)(1) \cdot h \\
        &= \hor_{A'}\big(f \circ \alpha, 0, \xi_0'\big)(1) \cdot kh \\
        &= \Phi_{(\pi', A')}\big(\llbracket f_*[\alpha], kh \rrbracket_{H'}\big) \\
        &= \Phi_{(\pi', A')}\Big(F_{f,k}\big(\llbracket [\alpha], h \rrbracket_H\big)\Big).
    \end{align*}
    Hence $\Phi : \mathcal K \mathcal H \Rightarrow \mathrm{Id}_{\BundleConnection^G}$ is a natural transformation, and in fact a natural isomorphism.

    Finally, with unit the identity $\mathrm{Id}_{\Holonomy^G} \Rightarrow \mathcal H \mathcal K$ and counit $\Phi : \mathcal K \mathcal H \Rightarrow \mathrm{Id}_{\BundleConnection^G}$, the zig-zag identities read
    \[
    \Phi_{\mathcal K(X, x_0, H)} = \id_{\mathcal K(X, x_0, H)} \quad \text{and} \quad \mathcal H(\Phi_{(\pi, A)}) = \id_{\mathcal H(\pi, A)}
    \]
    for all objects $(X, x_0, H)$ of $\Holonomy^G$ and $(\pi, A)$ of $\BundleConnection^G$. For the first, note that $\mathcal H\mathcal K(X, x_0, H) = (X, x_0, H)$ and that the base point of $\mathcal K(X, x_0, H)$ is $\xi_0^H = \llbracket [c_{x_0}], e \rrbracket_H$. Then,
    \[
    \Phi_{\mathcal K(X, x_0, H)}\big(\llbracket [\alpha], h \rrbracket_H\big) = \hor_{\Theta^H}(\alpha, 0, \xi_0^H)(1) \cdot h = \llbracket [\alpha], h \rrbracket_H.
    \]
    For the second, the fact that $\Phi_{(\pi, A)}(\xi_0^H) = \xi_0$ says that $k_{\Phi_{(\pi, A)}} = e$, so $\mathcal H(\Phi_{(\pi, A)}, \id_X) = (\id_X, e)$ is the identity morphism of $\mathcal H(\pi, A)$.

\end{proof}

\begin{remark}
    \label{rem:ImageOfK}
    By naturality of $\Phi : \mathcal K \mathcal H \Rightarrow \mathrm{Id}_{\BundleConnection^G}$, every $\BundleConnection^G$-morphism from $\mathcal K(X, x_0, H)$ to $\mathcal K(X', x_0', H')$ is of the form $\mathcal K(f, g) = (F_{f,g}, f)$. By the proof above, every object of $\BundleConnection^G$ is isomorphic to an object of the form $\mathcal K(X, x_0, H) = (\tau_H : (\mathcal F_{x_0}(X) \times_H G, \xi_0^H) \to (X, x_0), \Theta^H)$. This is the precise sense in which the pointed path bundle $\tau : \mathcal F_\bullet(X) \to X$ with its connection $\Theta$ is \emph{universal}: every bundle-connection pair $(\pi, A)$ is associated to $(\tau, \Theta)$ along its own holonomy representation, naturally in $(\pi, A)$.
\end{remark}

\begin{remark}
    \label{rem:Strict}
    One may instead work with the wide subcategories $\Holonomy^{G, \mathrm{st}} \subseteq \Holonomy^G$ and $\BundleConnection^{G, \mathrm{st}} \subseteq \BundleConnection^G$ whose morphisms are, respectively, the pairs $(f, g)$ with $g = e$ (that is, pointed smooth maps $f$ with $H' \circ f_* = H$) and the pairs $(F, f)$ with $F(\xi_0) = \xi_0'$ (that is, with $k_F = e$). By Lemmas \ref{lem:HolonomyOfMorphism} and \ref{lem:InducedBundleMap}, the functors $\mathcal K$ and $\mathcal H$ restrict to these subcategories, and, since $\Phi_{(\pi, A)}(\xi_0^H) = \xi_0$, so do the components of $\Phi$. In particular, Theorem \ref{thm:CategoryEquivalence} restricts to an equivalence $\Holonomy^{G, \mathrm{st}} \simeq \BundleConnection^{G, \mathrm{st}}$. However, as the anonymous referee appropriately pointed out, $\Holonomy^{G, \mathrm{st}}$ is ``rather rigid'' in view of Corollary \ref{cor:Classification}, which was a statement involving \textit{conjugation}. Within this stricter setting, the objects $(\pi : (E, \xi_0) \to (X, x_0), A)$ and $(\pi : (E, \xi_0 g) \to (X, x_0), A)$ need not be isomorphic. Indeed, a strict isomorphism $(F, f)$ from the first to the second has $f$ a pointed diffeomorphism and, by Lemma \ref{lem:HolonomyOfMorphism}, satisfies $g\inverse H^A_{\xi_0}(f_*[\gamma]) g = H^A_{\xi_0}([\gamma])$ for all $[\gamma] \in \Omega_{x_0}(X)$. Since $f_*$ is bijective on $\Omega_{x_0}(X)$, this forces $g\inverse S g = S$ for the subgroup $S = H^A_{\xi_0}(\Omega_{x_0}(X)) \leq G$, which fails of course whenever $g$ does not normalize $S$.
\end{remark}

\bibliographystyle{plainnat}
\bibliography{references}
\end{document}